\documentclass[reqno]{amsart}
\usepackage{amssymb}
\usepackage{amsmath}
\usepackage{url}

\makeatletter
\@addtoreset{equation}{section}
\makeatother

\renewcommand\thefigure{\thesection.\@arabic\c@figure}
\renewcommand\thetable{\thesection.\@arabic\c@table}

\newtheorem{theorem}{Theorem}[section]

\newtheorem{lemma}[theorem]{Lemma}
\newtheorem{proposition}[theorem]{Proposition}

\newtheorem{remark}[theorem]{Remark}

\def\R{\mathbb R}
\def\N{\mathbb N}

\def\N{\mathbb N}

\def\w{{\mathcal W}}

\def\C{\mathbb C}
\def\c{{\tilde c}}

\renewcommand{\S}{\mathcal{S}}

\begin{document}

\author{Jihyeok Choi, Sunder Sethuraman,\\
 and Shankar C. Venkataramani}

\address{\noindent Department of Mathematics, Syracuse University,
  Syracuse, NY  13244
\newline
e-mail:  \rm \texttt{jchoi46@syr.edu}
}

\address{\noindent Department of Mathematics, University of Arizona,
  Tucson, AZ  85721
\newline
e-mail:  \rm \texttt{sethuram@math.arizona.edu}
}

\address{\noindent Department of Mathematics, University of Arizona,
  Tucson, AZ  85721
\newline
e-mail:  \rm \texttt{shankar@math.arizona.edu}
}

\title[Scaling limit in sublinear preferential attachment schemes]{A scaling limit for the degree distribution in sublinear preferential attachment schemes}

\begin{abstract}
We consider a general class of preferential attachment schemes 
evolving by
a reinforcement rule with respect to certain sublinear weights.  In
these schemes, which grow a random network, the sequence of degree
distributions is an object of interest which sheds light on the
evolving structures.

In this article, we use a fluid limit approach
to prove a functional law of large
numbers for the degree structure in this class, starting from a variety of initial conditions.  The method appears robust and applies in particular to `non-tree' evolutions where cycles may develop in the network.

A main part of the argument is to analyze an infinite
system of coupled ODEs, corresponding to a rate formulation of the law of large numbers limit, in terms of $C_0$-semigroup/dynamical systems methods.    These results also resolve a question in Chung, Handjani and Jungreis (2003).

  \end{abstract}



\maketitle

\section{Introduction}
Since the late 90's and early 2000's, much attention has been devoted to `preferential attachment processes':  Networks evolving over time by linking
at each time step new nodes to vertices in the existing graph with a probability
based on their connectivity.  Such schemes relate to `reinforcement' and other dynamics which have a long history (cf.
surveys \cite{Mitzenmacher}, \cite{Pemantle}, \cite{Simkin}).
Recently, Barab\'{a}si and Albert (BA) in
\cite{Albert-Barabasi-99} proposed that versions of these processes may serve as
models for growing
real-world networks such as the world wide internet web, and types of social
structures.

For instance, in a `friend network', a newcomer may be favorably disposed to link or
become friends with an individual with high connectivity, or in other words, one who
already has many friends.  As observed in \cite{Albert-Barabasi-99}, when the probability of selecting a vertex is proportional to
its degree, the proportions of nodes with degrees $1, 2, \ldots,
k, \ldots$ converge as time grows to a power-law distribution $\langle q(k): k\geq 1
\rangle$ where
$0<\lim_{k\uparrow\infty} q(k)k^{\theta} <\infty$ for
some $\theta>0$.  Since the sampled empirical degree structure in many real-world networks also has such a power-law form, such preferential attachment processes, in contrast to Erd\"os-R\'enyi graphs where the degree structure decays much more rapidly, have become popular:  See \cite{Albert-Barabasi-02}, \cite{Barabasi}, \cite{BS},
\cite{Cald}, \cite{Chung-Lu}, \cite{CH}, \cite{DM}, \cite{Durrett_book},
\cite{Newman}, \cite{Newman10}, \cite{NW}, and references therein.

At the same time, other versions of preferential attachment, where the selection probability is a nonlinear function of the connectivity have been considered, and interesting effects have been shown:  See, among other works, \cite{CHJ}, \cite{dereich-2009}, \cite{Drinea-Frieze-Mitzenmacher}, \cite{KR}, \cite{Oliveira}, \cite{RTV}.  For instance, depending on the scheme and the type of nonlinearity, the degree structure asymptotically may be in the form of a `stretched exponential' or the graph may evolve into a `condensed' state in which a single (random) vertex may be linked with almost all the incoming nodes.

To be more specific, consider the following preferential attachment model.  Suppose at time $n=0$, the initial network $G_0$ is composed of two vertices with a
single (undirected) edge between them.  The dynamics now is that at time $n=1$, a new vertex is attached to one of the two
vertices in $G_0$ with probability proportional to a function of its degree to
form the new network $G_1$.  This scheme continues:  More precisely, at time $n+1$,
a new node is linked to vertex $x\in G_n$ with probability proportional to
$w(d_x(n))$, that is chance $w(d_x(n))/\sum_{y\in G_n}w(d_y(n))$, where $d_z(n)$ is
the degree at time $n$ of vertex $z$ and $w=w(d): \N \rightarrow \R_+$ is the
`weight' function.

Now, for the moment, to simplify the discussion, let us assume $w(d) = d^\kappa$ for $\kappa>-\infty$.
In this way, since the initial graph is a tree, all later networks $G_n$ for $n\geq
0$ are also trees.
Let now $\mathcal{Z}_k(n)$ be the number of vertices in $G_n$ with degree $k$,
$\mathcal{Z}_k(n) = \sum_{y\in G_n} 1(d_y(n) = k)$.  In \cite{KR}, a trichotomy of growth behaviors was observed depending on the strength of the exponent $\kappa$.

 First, when $w$ is linear, that is when $\kappa =1$, the scheme is the well known Barabasi-Albert model where the degree structure satisfies, for $k\geq 1$,
 $$\lim_{n\rightarrow\infty} \frac{\mathcal{Z}_k(n)}{n} \ = \ \frac{4}{k(k+1)(k+2)} \ \ {\rm a.s.}$$
This power-law ($\theta =3$), in mean-value, through an analysis of rates, was found in \cite{Albert-Barabasi-99}, \cite{KR}.  In \cite{BRST-01}, using difference equations/concentration bounds, the limit was proved in probability.  Via P\'olya urns, another proof was found yielding a.s. convergence, and also central limit theorems \cite{Mori-01}.  Also, by embedding into continuous time branching processes, the same a.s. limit was proved in \cite{RTV}; see also \cite{AGS} where a different type of embedding was used.  A form of Stein's method gives rates of convergence in total variation norm \cite{PRR}, \cite{Ross}.  A large deviation approach also obtains the limit \cite{CS}.

Next, in the strict sublinear case, when $\kappa<1$,
it was shown that
\begin{equation}
\label{sublinearLLN}
\lim_{n\uparrow\infty}\mathcal{Z}_k(n)/n \ =\  q(k) \ \ \ {\rm a.s.}\end{equation}
 although
$q$ is not a
power law, but in form where it decays faster than any polynomial \cite{KR},
\cite{RTV}: For $k\geq 1$,
\begin{equation*}
\label{stretched}
q(k) \ = \ \frac{s^*}{k^\kappa}\prod_{j=1}^k \frac{j^\kappa}{s^* + j^\kappa}, \ \
{\rm and \  }s^* {\rm \ is \ determined\ by \ \ } 1\ = \ \sum_{k=1}^\infty \prod_{j=1}^k
\frac{j^\kappa}{s^* + j^\kappa}.\end{equation*}
Asymptotically as $k\uparrow\infty$, when $0<\kappa<1$, $\log q(k) \sim -(s^*/(1-\kappa))k^{1-\kappa}$ is in `stretched exponential' form; when $\kappa<0$, 
$\log q(k)\sim \kappa k \log k$; when $\kappa=0$, the case of uniform attachment when an old vertex is
selected uniformly, $s^*=1$ and $q$ is geometric: $q(k) = 2^{-k}$ for
$k\geq 1$.

 In the superlinear case, when $\kappa>1$, `explosion' or a sort of
`condensation' happens in that in the limiting graph a random
single vertex dominates in accumulating connections.
In particular,
the
limiting graph is shown to be a tree
\begin{eqnarray}
&&{\rm  where\  there\  is\  a\  single\  random\  vertex\  with\  an\  infinite\
number\  of\  children;} \nonumber \\
&&{\rm all \
other\  vertices\  have\  bounded\  degree, \ and\  of\  these\  only\  a\
finite\  number}\nonumber\\
&&{\rm  have\  degree\  equal \ or \ larger\  than\  }\lfloor \kappa/(\kappa-1)\rfloor
\label{superlinear}\end{eqnarray}
 (cf. for a more precise
description \cite{Oliveira}, \cite{KR}).
Moreover, a corresponding LLN limit,
$\lim_{n\uparrow\infty}E\mathcal{Z}_k(n)/n = q(k)$, is proposed where $q$ is
degenerate in that $q(1)=1$ but $q(k)=0$ for $k\geq 2$ (cf. \cite{KR},
\cite[Chapter 4]{Durrett_book}, \cite{A}).

We now comment on the methods in the papers \cite{RTV} and \cite{Oliveira}.  Both use branching process embedding techniques to establish the sublinear and superlinear degree structure results \eqref{sublinearLLN} and \eqref{superlinear}.  More specifically, it seems a tree structure is useful in the proofs, that is the dynamics places no edges between already extant vertices to create cycles.

The purpose of this article, in this context, is to show the LLN for the degree structure in a general class of `sublinear'
preferential attachment models, including the scheme discussed above, starting from various initial conditions, through a new, different `fluid limit' approach
where cycles may
develop.   We also note that the method taken here seems robust and might be used in other combinatorial schemes (cf. Remark \ref{rmk}).

Specifically, we show (Theorem
\ref{mainthm}) a functional LLN for the degree counts in
sublinear generalizations of the urn scheme of Chung, Handjani and
Jungreis \cite{CHJ} and the graph model of Chung and Lu \cite{Chung-Lu}
(cf. Section \ref{models} for model descriptions and assumptions).    Moreover, our
work solves a question in \cite{CHJ} to show a LLN for the
associated degree structure when the weights are sublinear (cf. Remark
\ref{rmk}).

The `fluid limit' method is to consider a more complex problem, namely that of the dynamics of paths $\{n^{-1}\mathcal{Z}_k(\lfloor nt\rfloor): t\in [0,1]\}$ for $k\geq 1$.  But, these paths have nice properties and we show their limit points satisfy certain ODEs corresponding to a rate formulation of the degree distribution flow (cf. \eqref{ODE}).  
As all the counts $\{\mathcal{Z}_k(n)\}$ are coupled together in terms of the total `weight' of the graph $S(n) = \sum_{k\geq 1}\mathcal{Z}_k(n)$ in the selection procedure, the ODE system derived is infinite dimensional and nonlinear, and poses nontrivial difficulties.  

The ODEs appear natural and may be of interest in other contexts where there is exchange of proportional flow between chains of components.  
By a change of variables, the ODEs can be written in terms of a linear `Kolmogorov' differential equation (cf. \cite{thieme2}) which can be analyzed by $C_0$-semigroup/dynamical systems arguments.  In particular,
we show (Theorem \ref{uniqueness}) the
ODEs admit a unique solution.  Therefore, addressing the original `fluid limit' taken, all the path limit
points are the same and so are uniquely characterized.  

There is a large literature on fluid limits in various contexts:  See \cite{Chen_Yao}, \cite{DN}, \cite{Foss}, \cite{Robert}, \cite{Whitt}, \cite{Wormald}) and references therein.  Most of this previous development focuses on finite dimensional spaces.  In this respect, the current article considers a nontrivial infinite dimensional fluid limit, whose analysis depends on the type of initial condition, namely `small' versus `large' (cf. (LIM) in Section 2), which plays a role in the results Theorems \ref{uniqueness}, \ref{mainthm}.  See also \cite{RT} for a different infinite dimensional limit in a type of Erd\"os-R\'enyi graph.

In the next Section, we detail the preferential attachment
models discussed, and state results.
     Then, proofs of the main convergence and uniqueness results follow in succeeding Sections.

\section{Models and Results}
\label{models}
To specify the models considered, let $0\leq p \leq 1$ be a parameter, and let $w:\{1, 2, \ldots\}\rightarrow (0,\infty)$ be a positive function which we will call the `weight' function.

\medskip
{\it Graph Model.}
The following scheme captures the growth of a graph network:
\begin{itemize}
\item  At time $n=0$, the initial network $G_0$ is a finite, possibly disconnected graph.
\item At time $n+1\geq 1$, form $G_{n+1}$ as follows.

--With probability $1-p$, we select independently two old vertices $x,y\in G_n$ with chances $w(|x|)/S(n)$ and $w(|y|)/S(n)$ respectively, and then place an edge connecting $x$ and $y$ to form $G_{n+1}$.

--However, with probability $p$, an edge is placed between a new vertex and an old node $x\in G_n$, chosen with probability $w(|x|)/S(n)$, to form $G_{n+1}$.
\end{itemize}
In this model, $|x|\geq 1$ is the degree of the vertex $x$, and $S(n)$ is the total `weight' of collection $G_n$:
$$S(n) \ = \ \sum_{k\geq 1}w(k)Z_k(n)$$
where $Z_k(n)$ is the count of vertices in $G_n$ with degree $k$ for $k\geq 1$.

Note that it may be possible when two vertices $x,y\in G_n$ are selected, they are the same, which means a `loop' is added to the graph at $x=y$, and our convention here is that the degree of vertex $x=y$ is incremented by $2$.  In particular, at each time, the total degree of the graph increments by $2$.  However, as the successive independent choices in actions are random, the total number of vertices at time $n\geq 1$ is $V(G_0)$ plus a sum of $n$ independent Bernoulli$(p)$ variables; here, $V(G_0)$ is the initial number of vertices.

Cases of the above dynamics include the following. 
\begin{itemize}
\item When $p=1$ and $w(k)=k^\kappa$, the dynamics matches the
  preferential attachment graph scheme mentioned in the Introduction,
  and $Z_k(n) = \mathcal{Z}_{k}(n)$ for all $k\geq 1$.  In this case,
  for $\kappa\leq 1$, a LLN limit for $Z_k(n)/n$ has been proved, among other results, in \cite{RTV} as mentioned before.

\item When $w(k)=k^\kappa$ and $p$ is arbitrary, the scheme is discussed and many results
  are proved in \cite{Chung-Lu}.  For instance, when $\kappa = 1$, a LLN
  is proved, $\lim_{n\uparrow\infty} Z_k(n)/n = q(k)$ where $q$ is in
  `power law' form with $\theta = 1+2/(2-p)$.  Also, in this
  situation, a central limit theorem for the `leaves', nodes of degree
  $1$, has been found in \cite{W}.  However, when $\kappa<1$ and $p<1$, the
  LLN for $Z_k(n)/n$ has been an open question, now resolved by Theorem \ref{mainthm}. 

\item If $p=0$, the dynamics would always add edges and loops with respect to the initial graph $G_0$.  We will avoid this `degenerate' growth in what follows.
\end{itemize}

\medskip

{\it Urn Model.}
Consider the following urn dynamics which builds an evolving collection of urns:
\begin{itemize}
\item  At time $n=0$, the initial collection $G_0$ is a finite set of nonempty urns, each containing a finite number of balls.
\item At time $n+1\geq 1$, $G_{n+1}$ is built as follows.

--With probability $p$, a new urn with a single ball is added to the collection to form $G_{n+1}$.

--However, with probability $1-p$, we select an urn $x$ from $G_{n}$ with probability $w(|x|)/S(n)$, and place a new ball into it to form $G_{n+1}$.
\end{itemize}
Here, $|x|\geq 1$ is the size or number of balls in the urn $x$, and $S(n)$ is the total `weight' of collection $G_n$:
$$S(n) \ = \ \sum_{k\geq 1}w(k)Z_k(n)$$
where $Z_k(n)$ is the number of urns in $G_n$ with exactly $k$ balls for $k\geq 1$.

We note that the total size or number of balls increments by $1$ at each time, but as in the graph model, at time $n\geq 1$, the total number of urns is random, namely $U(G_0)$ plus the sum of $n$ independent Bernoulli$(p)$ variables, where $U(G_0)$ is the initial number of urns.

We now remark on some cases of the above dynamics:  
\begin{itemize}

\item When $w(k)=k^\kappa$, the scheme is discussed in \cite{CHJ}, and
  results on the evolution are given when $\kappa\geq 1$.  However,
  when $\kappa<1$, a LLN is stated, but the convergence of $Z_k(n)/n$
  is left open (cf. Remark \ref{rmk}).

\item If $p=1$, the scheme would always add an urn with a single ball to the collection at each time.  Also, if $p=0$, no new urns are added and only the urns in the initial collection grow.  Both are `degenerate' evolutions which we will avoid in assumption (P) below.
\end{itemize}

\medskip

We now give assumptions on $p$ with respect to the two models, and on the weight function $w$ under which results are stated.
\begin{itemize}
\item[(P)] To avoid `degeneracies', in the graph model, $p$ is taken $0<p\leq 1$.  However, in the urn scheme, we assume that $0<p<1$.
\end{itemize}

\begin{itemize}
\item[(SUB)] We have
$$\lim_{k\uparrow\infty} \frac{w(k)}{k} =  0.  \ \ \ {\rm Hence, \ } \sup_{k\geq 1} \big[w(k)/{k}\big] \leq \w \ {\rm for \ some \ constant \ }\w<\infty.$$
\end{itemize}
This large class of weights $w(\cdot)$ includes in particular the well-studied case $w(k) = k^\kappa$ for $\kappa<1$ discussed in the Introduction.  

 Let $p_0,q_0>0$.  In the following, with respect to the two models under (P), 
 \begin{itemize}
\item[]  fix  in  the  graph  scheme $p_0=p$ and $q_0=2-p$, and
\item[] in  the  urn  scheme $p_0=p$ and $q_0=1-p$.
 \end{itemize}
   Define also the positive function
 $F_{p_0,q_0}:[0,\infty) \rightarrow \R\cup\{\infty\}$ by 
\begin{equation*}
F_{p_0,q_0}(s) \ = \ \frac{p_0}{q_0}\sum_{k\geq 1}\prod_{j=1}^k \frac{q_0w(j)}{s+ q_0w(j)}.
\end{equation*}

By (SUB), $F_{p_0,q_0}(s)<\infty$ for $s>0$.  Moreover, it is clear in each model that
there exists
a number $s_0=s_0(p_0,q_0)>0$ such that
\begin{equation}
\label{RTV_eqn}
1 \ < \ F_{p_0,q_0}(s_0) \ < \ \infty.\end{equation}
Given \eqref{RTV_eqn}, the function $F_{p_0,q_0}$ is strictly decreasing on $[s_0,\infty)$ and
vanishes at infinity, $\lim_{s\uparrow\infty}F_{p_0,q_0}(s) = 0$.  Therefore,
\begin{equation}
\label{s*defn}
{\rm there \ exists \ a \ unique \ } s^*=s^*(p_0,q_0)>s_0 {\rm \  such \ that \ }F_{p_0,q_0}(s^*)=1.\end{equation}

We now comment, in \cite{RTV}, which proves the LLN for the degree distribution in the graph model when $p=1$, the only condition assumed on $w$ is that \eqref{RTV_eqn} holds, which is more general than (SUB).  For instance, a `linear weights'-type structure where $\liminf_{k\uparrow\infty}w(k)/k>0$ is not allowed under (SUB) (although see Remark \ref{rmk}).
However, as remarked in \cite{RTV}, \eqref{RTV_eqn} itself is only given as a sufficient condition.  

A necessary condition might include the requirement, $\sum_{k\geq 1}1/w(k) = \infty$, implied by \eqref{RTV_eqn}, although this is not pursued here.  In this respect, we note Theorem 1.1(ii) in \cite{Collevechio} proves that $\sum_{k\geq 1}1/w(k)=\infty$ is a necessary and sufficient condition for all vertices/urns to have infinite size in the limit network a.s.

The assumption (SUB) is used to enforce control on the
tails of the weight sum $S(n)$, so that the limits of
$S(n)/n$ and $\{Z_k(n)/n\}$ can be related (cf. Step 2, proof of Theorem
\ref{mainthm}). (SUB) is also useful in the proof of
uniqueness of solution to the infinite dimensional ODE system derived
(cf. Section \ref{ODE_uniqueness}).

\medskip

We now derive the evolution scheme of the counts $\{Z_k(n)\}$ for $n\geq 0$.  Define sigma-fields
$\mathcal{F}_j = \sigma\{\{Z_k(\ell)\}: 0\leq \ell \leq j\}$ for $j\geq 0$.  Note also, in both models, as $w(\cdot)>0$, $S(n)>0$ for all $n\geq 0$.  For each $k\geq 1$, let
\begin{equation}
\label{mart_decomp}
Z_k(j+1) - Z_k(j)\ =: \ d_k(j+1)\end{equation}
where given $\{Z_k(j)\}$, the counts at time $j$, the difference
$d_k(j+1)$ is as follows.
\medskip

With respect to the urn scheme,
$$d_1(j+1) \ = \ \left\{ \begin{array}{rl}
1& \ \ {\rm with \ prob. \ } p \\
-1& \ \ {\rm with \ prob. \ }(1-p)\frac{w(1)Z_1(j)}{S(j)}\\
0& \ \ {\rm with \ prob. \ } (1-p)\big[1- \frac{w(1)Z_1(j)}{S(j)}\big].
\end{array}\right.
$$
Also, for $k\geq 2$, 
$$d_k(j+1) \ = \ \left\{\begin{array}{rl}
1& \ \ {\rm with \ prob. \ } (1-p)\frac{w(k-1)Z_{k-1}(j)}{S(j)}\\
-1& \ \ {\rm with \ prob. \ } (1-p)\frac{w(k)Z_k(j)}{S(j)}\\
0& \ \ {\rm with \ prob. \ } 1-(1-p)\big[\frac{w(k-1)Z_{k-1}(j)}{S(j)} + \frac{w(k)Z_k(j)}{S(j)}\big].
\end{array}\right. $$

In the graph scheme, $d_k(j+1)$ has a similar, but more
involved expression as possible loops need to be considered. 
These formulas are given in the Appendix.
\medskip

Consider now an array of counts $\{Z^n_k(\cdot): k\geq 1\}$ and weights $\{S^n(\cdot)\}$ for $n\geq 1$ where in the $n$th row the underlying process begins from initial network $G^n_0$.  Define the family of linearly interpolated processes $\{X^n_k(t): t\geq 0, k\geq 1\}$ for $n\geq 1$, which place the proportion of counts $\{Z^n_k(j)/n\}$ into continuous time trajectories, by
$$X^n_k(t) \ := \ \frac{1}{n}Z^n_k(\lfloor nt\rfloor) + \frac{nt - \lfloor nt\rfloor}{n} \left( Z^n_k(\lceil nt\rceil) - Z^n_k(\lfloor nt\rfloor)\right).$$
The paths $X^n_k:[0,\infty)\rightarrow \R_+$, in both schemes as $|d_k(j+1)|\leq 2$, belong to the space of Lipschitz functions with constant at most $2$.
For $t\geq 0$, let also $\S^n$ be the continuous interpolation of the weights $S^n$,
\begin{equation}\label{weight_equation}
\S^n(t) \ := \ \frac{1}{n}S^n(\lfloor nt\rfloor) + \frac{nt - \lfloor nt\rfloor}{n} \left( S^n(\lceil nt\rceil) - S^n(\lfloor nt\rfloor)\right).
\end{equation}

Let now, initially, for some constants $c_k^n, c^n, \c^n\geq 0$ for $k\geq 1$ that
\begin{eqnarray*}
c_k^n &:=& \frac{1}{n}Z^n_k(0), \ \ c^n\ :=\  \sum_{k\geq 1}c_k^n,
   \ \ {\rm and \ \ } \c^n:= \sum_{k\geq 1}kc^n_k.
\end{eqnarray*}
We will impose the following initial laws of large numbers:
\begin{itemize}
\item[(LIM)]  For constants $c_k, c, \c \geq 0$, we have
$$c_k := \lim_{n\uparrow\infty}c_k^n \ \ {\rm and \ \ } \c:= \sup_{n\geq 1} \c^n <\infty.$$
Hence, $c:= \lim_{n\uparrow\infty} c^n = \sum_{k\geq 1} c_k<\infty$ and $\c \geq \sum_{k\geq 1}kc_k$.
\end{itemize}
Here, noting $\sum_{k\geq L}c_k^n \leq L^{-1}\sum_{k\geq 1}kc^n_k \leq \c/L$, the $c^n$-limit follows.  The $\c$-inequality is Fatou's inequality.    
Also,
as 
$\sum_{k\geq L}w(k)c_k^n \leq \c[\sup_{j\geq L}w(j)/j]$, from (SUB), we may also conclude the initial weights $\sum_{k\geq 1} w(k)c^n_k \rightarrow \sum_{k\geq 1}w(k)c_k$. 

We will say that the initial configuration is a `small' configuration if $c_k\equiv 0$, and is a `large' one if $c_k>0$ for some $k\geq 1$.  In a small configuration, the total degree/size of $G^n_0$ is $o(n)$.  In particular, if the total degree/size of $G^n_0$ is uniformly bounded, for example $G^n_0 \equiv G_0$ is fixed, the initial configuration is a small one.   However, in a large configuration, the initial networks are already developed in that their degree/size is at least $\varepsilon n $ for some $\varepsilon >0$.  We remark similar initial conditions were used in a different context in \cite{CS}.

With respect to small initial configurations, we now try to guess the long term behavior of $X^n_k(t)$ for all $k$.  Suppose, as $n\uparrow\infty$, that $X^n_k(t) \rightarrow a_kt$ and $\S^n(t)\rightarrow bt$ for constants $\{a_k\}$ and $b=\sum_{k\geq 1}w(k)a_k >0$.
Then, since $Z_k(\lfloor nt\rfloor)$ is a type of random walk, considering the drift, if the increment $d_k(j+1)$ were set equal to its conditional expectation given $\mathcal{F}_j$, we would obtain heuristically the following limit equations for $\{a_k\}$ in the two models discussed.

Let $p_0,q_0>0$.  In the following, under (P), fix $p_0=p$ and $q_0 = 2-p$, the parameters correspond to the the graph model.  However, in the urn model, set the parameters $p_0=p$ and $q_0=1-p$.

Then, for $b>0$, define $a_k = a_k(p_0,q_0,b)$ by
\begin{eqnarray}
\label{discrete_ODE}
a_1&=& p_0 - \frac{q_0w(1)a_1}{b}\\
a_k&=& q_0\frac{w(k-1)a_{k-1}}{b} - q_0\frac{w(k)a_k}{b} \ \ \ \ \ \ \ \ \ \ \ \ {\rm for \ }k\geq 2.
\nonumber
\end{eqnarray}
 Solving for the $\{a_k\}$, in terms of $b$, one obtains
\begin{equation*}
\label{a_eqn}
a_1 = \frac{bp_0}{b+ q_0w(1)},  {\rm \ and \ } \ a_k = \frac{q_0w(k-1)a_{k-1}}{b+ q_0w(k)} \ \ {\rm for \ }k\geq 2.\end{equation*}
Therefore, for $k\geq 2$,
\begin{eqnarray*}
a_k& = & a_1\prod_{j=1}^{k-1}\frac{q_0w(j)}{b+ q_0w(j+1)}\\
&=&a_1\frac{b+q_0w(1)}{b+ q_0w(k)}\prod_{j=1}^{k-1}\frac{q_0w(j)}{b+ q_0w(j)}\nonumber\\
&=&a_1\frac{b+q_0w(1)}{b}\Big[\prod_{j=1}^{k-1} \frac{q_0w(j)}{b+q_0w(j)} - \prod_{j=1}^k \frac{q_0w(j)}{b + q_0w(j)}\Big].\nonumber
\end{eqnarray*}

Noting $\{a_k\}$ above, under (SUB), the equation $b=\sum w(k)a_k$ takes form
$$b \ = \ \frac{bp_0}{q_0}\sum_{k\geq 1}\prod_{j=1}^k \frac{q_0w(k)}{b+ q_0w(k)}$$
 or $1 = F_{p_0,q_0}(b)$ (cf. definition near \eqref{RTV_eqn}).

Hence, the parameter $b$ above is identified from \eqref{s*defn} as $b=s^*$, which through \eqref{discrete_ODE} implicitly determines $\{a_k\}$ in terms of $w(\cdot)$, $p_0$ and $q_0$.  Of course, when $w(k)= k^\kappa$ for $\kappa<1$, $a_k = q(k)$ for $k\geq 1$ (cf. \eqref{sublinearLLN}).

In both models, after a calculation,
one sees 
\begin{equation}
\label{sums}
\sum_{k\geq 1} a_k = p_0 \ \ {\rm and \ \ } \sum_{k\geq 1}ka_k = p_0 + q_0.\end{equation}

\medskip

At this point, one might also infer an a.s. `continuous' version of \eqref{discrete_ODE}, a rate formulation for the limit of the functions $\{X^n_k\}$.
That is, for each realization in a probability $1$ set, we take subsequential limits of $X^n_k$ and conclude that all limit points are nonnegative functions $\varphi_k(\cdot)=\varphi_k(\cdot;p_0,q_0)$ for $k\geq 1$ satisfying the integral form of a coupled system of ODEs:
    \begin{eqnarray}
    \label{ODE}
\dot\varphi_1(t)
    & =& p_0 -\frac{q_0w(1)\varphi_1(t)}{\sum_{k\geq 1}w(\ell)\varphi_k(t)},\\
\dot\varphi_k(t)
    & =& \frac{q_0}{\sum_{k\geq 1}w(\ell)\varphi_k(t)}\big[w(k-1)\varphi_{k-1}(t) -w(k)\varphi_k(t)\big], \ \ {\rm for \ }k\geq 2.\nonumber
    \end{eqnarray}
with initial condition $\varphi_k(0)=c_k$ for $k\geq 1$.  Under small initial configurations ($c_k\equiv 0$), the ODE is singular at $t=0$, although one can inspect that $\{\varphi_k(t) = a_k(p_0,q_0,s^*)t: k\geq 1\}$ is a solution since $\{a_k(p_0,q_0,s^*): k\geq 1\}$ verifies \eqref{discrete_ODE}.  However, under either large or small initial configurations, it does not seem easy to conclude the ODEs has a unique nonnegative solution.

But, one can think of
the ODE system in the following way:  Introduce a
time-change $t=t(s)$ satisfying $\dot t = T(t)$ where $T(t) =
\sum_{k\geq 1}w(k)\varphi_k(t)$ and $t(0)=1$.  Then, $\{\psi_k(s) =
\varphi_k(t(s)): k\geq 1\}$ satisfies the integral form of the following autonomous system:
 \begin{eqnarray}
\label{psi_ODE}
\dot\psi_1(s)
    & =& (p_0-q_0)w(1)\psi_1(s) + p_0\sum_{k\geq 2}\psi_k(s),\\
\dot\psi_k(s)
    & =& q_0\big[w(k-1)\psi_{k-1}(s) -w(k)\psi_k(s)\big], \ \ \ \ \ \ {\rm for \ }k\geq 2.\nonumber
    \end{eqnarray}    
 
 It will turn out that one can associate to this system a strongly
 continuous positive semigroup $P_s$ whose essential growth rate is
 nonpositive.  Such semigroups have nice asymptotics in terms of a `Perron-Frobenius' eigenvector and eigenvalue.  Moreover, it turns out the time-change $t(\cdot)$ can be determined in terms of $P_s$, which will allow to characterize solutions of \eqref{ODE}.

\begin{theorem}
\label{uniqueness}
Suppose $p_0,q_0>0$ and conditions (SUB) and (LIM) hold, and recall the parameter $s^*$ in \eqref{s*defn}.  Then, under both small and large initial configurations,
there is a unique nonnegative solution $\{\varphi_k(\cdot)\}$ of the integral form of ODEs \eqref{ODE}.  

Moreover, for $k\geq 1$, under small initial configurations, $\varphi_k(t) = a_k(p_0,q_0,s*)t$.   Also, with respect to large initial configurations, $\lim_{t\uparrow\infty} t^{-1}\varphi_k(t) = a_k(p_0,q_0,s^*)$.
\end{theorem}

\begin{remark}
\label{uniqueness_rmk}
\rm

The proof of Theorem \ref{uniqueness} is
shorter under large initial configurations as there is no time
singularity at $t=0$.  In this case, the solution is found implicitly in terms of the semigroup $P_s$ and time-change $t=t(s)$.

However, the full machinery of `quasi-compact' semigroup
asymptotics and the assumption $c=0$ are used in the small initial
configuration case.  See the beginning of Section \ref{ODE_uniqueness} for more remarks on the strategy of the proof. 

 Finally, we note 
$s^*$ can be identified in terms of the `Perron-Frobenius' eigenvalue alluded to earlier (cf. Proposition \ref{adjoint_eigenvalue}).

\end{remark}

We now assert that the heuristic derivations \eqref{discrete_ODE} and \eqref{ODE} are correct.
\begin{theorem}
\label{mainthm}
Suppose conditions (P), (SUB), and (LIM) hold.   Let $\{\varphi_k(\cdot, p_0,q_0): k\geq 1\}$ be the unique nonnegative solution found in Theorem \ref{uniqueness} to the integral form of ODEs \eqref{ODE}.  Then, with respect to the graph and urn models, for $k\geq 1$, uniformly on compact time sets, we have

$$\lim_{n\uparrow\infty} X^n_k(t) \ = \ \varphi_k(t; p_0, q_0)\ \ \ {\rm a.s.}$$

\end{theorem}

    \begin{remark}\rm
\label{rmk}
With respect to the urn model, when $w(k)=k^\kappa$ for $\kappa<1$, the form of $\{a_k\}$ was derived in \cite{CHJ}.
However, it was left open in \cite{CHJ} to show that the LLN $\lim_{n\uparrow\infty} Z_k(n)/n = a_k$ holds for $k\geq 1$. In this context, a contribution of Theorem \ref{mainthm} is to give a proof of this limit.

The fluid limit argument given seems of potential use in other nonlinear preferential attachment schemes.  In particular, the approach should hold for models where at each time only a finite number of vertices/edges or balls/urns are added.  In this case, the differences $d_k(j+1)$ are still uniformly bounded in $k,j$ and the paths $X^n_k(t)$ will be Lipschitz, a primary ingredient in the proof.

 In addition, although (SUB) excludes the ``linear weights'' case $w(k)=k+m$ for $k\geq 1$ and $m>-1$, since in this case $S_n$ acts as an affine function of $n$, and the corresponding ODEs \eqref{ODE} can be uniquely integrated (cf. 
 Corollary 1.7 in \cite{CS}), a similar fluid limit argument yields yet another proof in this situation.  
\end{remark}

\section{Proof of Theorem \ref{mainthm}.}
\label{proofs_section}

We will assume Theorem \ref{uniqueness}, proved in the next Section, and prove Theorem \ref{mainthm} in several steps.

\medskip
\noindent {\it Step 1.} Since $d_k(j)$ is uniformly bounded,
$\|d_k(j)\|_{L^\infty} \leq 2$, we have, for each realization of the
evolving scheme, that $X^n_k$ are Lipschitz functions with constant
$2$ for all $k\geq 1$ and $n\geq 1$.  Since $X^n_k(0) = c^n_k$
converges to
$c_k$, by equicontinuity and local compactness of $[0,\infty)$, we may take a diagonal subsequence $n_m$ so that
$X^{n_m}_k$ converges uniformly for $t$ in compact subsets of $[0,\infty)$ to a Lipschitz function $\varphi_k$ with constant $2$, for each $k\geq 1$, which may depend on the realization:
$$\lim_{m\uparrow\infty} \sup_{t\in [0,1]}\big|X^{n_m}_k(t) - \varphi_k(t)\big| \ = \ 0.$$

\medskip
\noindent {\it Step 2.}
With respect to the graph model, as the total degree increments by $2$
at each time, we have $\sum_{k\geq 1}kZ^n_k(n)=n\c^n + 2n$.  On the
other hand, in the urn scheme, the total number of balls increases by
$1$ at each time, and so the total size $\sum_{k\geq 1}kZ_k(n)= n\c^n +
n$.  Hence, in both models,
given $\sum_{k\geq 1} kX^n_k(t) \leq \c^n +2t$, we have, for each $L\geq 1$, that
\begin{equation}\label{Fatou}
\sum_{k\geq L} w(k)X^{n_m}_k(t) \ \leq \ \big[\sup_{k\geq
  L}w(k)/k\big] \sum_{k\geq 1}kX^{n_m}_k(t) \ \leq \ \big[\sup_{k\geq
  L} w(k)/k\big](\c^n + 2t).
\end{equation}
Therefore, 
$$\sum_{k\leq L} w(k)X^{n_m}_k(t) \ \leq \ \S^{n_m}(t) \ \leq \
\sum_{k\leq L}w(k)X^{n_m}_k(t) + (\c^n + 2t)[\sup_{k>L}w(k)/k]$$
and also, for $N>0$, noting (LIM),
$$\lim_{n_m\uparrow\infty}\sup_{t\in [0,N]}\big| \S^{n_m}(t) -
\sum_{k\leq L}w(k)\varphi_k(t)\big| \ \leq \ (\c + 2N)[\sup_{k>L}w(k)/k].$$

In addition, by Fatou's lemma, from \eqref{Fatou}, we obtain $\sum
k\varphi_k(t) \leq \c+ 2t$.
Then,
$$\sup_{t\in [0,N]}\sum_{k>L}w(k)\varphi_k(t) \ \leq \ \big[\sup_{k>L}w(k)/k\big]\sup_{t\in [0,N]}\sum_{k\geq 1}k\varphi_k(t) \ \leq \ (\c+2N)\big[\sup_{k>L}w(k)/k\big].$$

Putting together these estimates, we have for each $L\geq 1$ that
$$\lim_{n_m\uparrow\infty}\sup_{t\in [0,N]} \big| \S^{n_m}(t) - \sum_{k\geq 1}w(k)\varphi_k(t)\big| \ \leq \ 2(\c+2N)[\sup_{k> L}w(k)/k].$$
Therefore, by assumption (SUB), taking $L\uparrow\infty$, we have
$$\S:=\lim_{m\uparrow\infty} \S^{n_m} \ = \ \sum_{k\geq 1}w(k)\varphi_k$$
converges uniformly for $t\in [0,N]$.
Since $\{\S^{n_m}\}$ are continuous functions, 
we see also that $\S$ is a continuous function.

\medskip
\noindent {\it Step 3.}
We now derive bounds for the limit function $\S$.
Under (SUB) and (LIM), in both models, given the bound $(\c + 2)n$ on the total degree/size of the network at time $n$,
we have
\begin{equation*}
\label{S_estimates}
\S^n(t) \ = \ \sum_{k\geq 1} w(k)X^n_k(t) \ \leq \ \w\sum_{k\geq 1}
kX^n_k(t) \ \leq \ (\c +2t)\w. \end{equation*}
Also, in both models, for $L\geq 1$,
we have
\begin{eqnarray*}\S^n(t) & \geq & \big(\inf_{k\leq L}w(k)\big)\sum_{k\leq L} X^n_k(t) \\
& \geq & \big(\inf_{k\leq L}w(k)\big)\Big[\sum_{k\geq 1} X^n_k(t)
- \frac{1}{L+1}\sum_{k\geq 1} k X^n_k(t)\Big] \\
& \geq & \ \big(\inf_{k\leq L}w(k)\big)\Big[\sum_{k\geq 1} X^n_k(t) -
\frac{\c + 2t}{L+1} \Big].\end{eqnarray*}

Since, in both models, at time $n\geq 1$, the number of vertices/urns at time $n$ equals $nc^n$ plus the sum of $n$ independent Bernoulli$(p)$ variables, we have
$\sum_{k\geq 1} X^{n_m}_k(t) \rightarrow c+ pt$ a.s. Therefore, with
$\hat L$ such that
$\c/(\hat L +1) \leq c/2$ and $2/(\hat L +1)\leq p/2$,
we conclude from the above estimates that
\begin{equation*}
\label{S_lower_estimates}
2^{-1}(c+pt)\big(\inf_{k\leq \hat L}w(k)\big) \ \leq \ \S(t) \ \leq \ (\c + 2t)\w .\end{equation*}

\medskip
\noindent {\it Step 4.}
From the martingale decomposition \eqref{mart_decomp}, we have, for $k\geq 1$, that
$$
X^n_k(t) - X^n_k(0)  =  M_k^n(\lfloor nt\rfloor) + \frac{1}{n}\sum_{j=0}^{\lfloor nt\rfloor -1} E\big[d^n_k(j+1)\big|\mathcal{F}_j\big] + \frac{nt-\lfloor nt\rfloor}{n} d_k(\lfloor nt\rfloor +1)
$$
where
$$M_k^n(\ell) \ = \ \frac{1}{n}\sum_{j=0}^{\ell -1} \Big(d_k^n(j+1) - E\big[d^n_k(j+1)\big|\mathcal{F}_j\big]\Big)$$
is a martingale with respect to $\{\mathcal{F}_\ell: \ell \geq 0\}$,
and,
for the urn scheme,
\begin{eqnarray*}
&&E\big[d^n_k(j+1)\big|\mathcal{F}_j\big] \\
&&\ \ \ \ \ \ \ \  = \ \left\{\begin{array}{ll}
 p-(1 -p)\frac{w(1)X^n_1(j/n)}{\S^n(j/n)} &  {\rm for \ }k=1\\
\frac{(1-p)w(k-1)X^n_{k-1}(j/n)}{\S^n(j/n)} - \frac{(1 -p)w(k)X^n_k(j/n)}{\S^n(j/n)}&  {\rm for \ }k\geq 2
\end{array}\right. \end{eqnarray*}
and, for the graph model, after calculating with $\{d^n_k\}$ in the Appendix,
\begin{eqnarray*}
&&E\big[d^n_k(j+1)\big|\mathcal{F}_j\big] \\
&&\ = \ \left\{\begin{array}{ll}
\begin{array}{ll}p-(2-p)\frac{w(1)X^n_1(j/n)}{\S^n(j/n)}\\
\ \ \ \ \ \ \ \ +\frac{1-p}{n}\Big\{\frac{[w(1)]^2X^n_1(j/n)}{[\S^n(j/n)]^2}\Big\} \end{array} &  {\rm for \ }k=1\\
\\
\begin{array}{ll}\frac{(2-p)w(1)X^n_{1}(j/n)}{\S^n(j/n)} - \frac{(2-p)w(2)X^n_2(j/n)}{\S^n(j/n)}\\
\ \ \ \ \ \ \ \ +\frac{1-p}{n}\Big\{\frac{-2[w(1)]^2X^n_1(j/n)}{[\S^n(j/n)]^2} +\frac{[w(2)]^2X^n_2(j/n)}{[\S^n(j/n)]^2}\Big\}\end{array} &  {\rm for \ }k=2\\
\\
\begin{array}{ll}\frac{(2-p)w(k-1)X^n_{k-1}(j/n)}{\S^n(j/n)} - \frac{(2-p)w(k)X^n_k(j/n)}{\S^n(j/n)}\\
\ \ \ \ \ \ \ \ +\frac{1-p}{n}\Big\{\frac{[w(k-2)]^2X^n_{k-2}(j/n)}{[\S^n(j/n)]^2}  +\frac{-2[w(k-1)]^2X^n_{k-1}(j/n)}{[\S^n(j/n)]^2} \\
\ \ \ \ \ \ \ \ \ \ \ \ \ \ \ \ \ \ \ \ \ \ \ \ +\frac{[w(k)]^2X^n_k(j/n)}{[\S^n(j/n)]^2}\Big\}\end{array} &  {\rm for \ }k\geq3,
\end{array}\right. \end{eqnarray*}

\medskip
\noindent {\it Step 5.}
Let $\langle M_k^n(j)\rangle$ be the quadratic variation of $M_k^n(j)$.  In our context, noting $|d_k^n(j)|\leq 2$ is bounded, we have
$$|\langle M_k^n(\lfloor nt\rfloor)\rangle| \ = \ \frac{1}{n^2}\sum_{j=0}^{\lfloor nt\rfloor -1} \big(d_k^n(j+1) - E[d_k^n(j+1)|\mathcal{F}_j]\big)^2\ \leq \ Ctn^{-1}.$$

Therefore, for $\epsilon>0$, by Burkholder-Davis-Gundy inequalities, we have
\begin{eqnarray*}
P\Big(\sup_{s\in [0,N]} |M_k^n(\lfloor ns\rfloor)| > \epsilon\Big) & \leq & \frac{1}{\epsilon^4} E\Big[\max_{0\leq j\leq \lfloor nN\rfloor}|M_k^n(j)|^4\Big]\\
&\leq& CE\Big[ \langle M_k^n(\lfloor nN\rfloor)\rangle^2\Big] \ \leq CN^2n^{-2}.
\end{eqnarray*}
Then, by Borel-Cantelli lemma, $\lim_{n\uparrow\infty}\sup_{t\in [0,N]}|M^n_k(\lfloor nt\rfloor)| = 0$ a.s.

\medskip
\noindent {\it Step 6.}  To obtain an integral equation, from the
development in Step 4, since $X^n_k(0)=c^n_k\rightarrow c_k$ by (LIM)
and $n^{-1}d_k(\lfloor nt\rfloor +1)$, $M_k^{n}(\cdot)$ vanish uniformly a.s., we need only evaluate the limit of
\begin{eqnarray*}
\frac{1}{n_m}\sum_{j=0}^{\lfloor n_mt\rfloor -1} E\big[d^{n_m}_k(j+1)\big|\mathcal{F}_j\big] &=&
\int_0^t E\big[d^{n_m}_k(\lceil n_m s\rceil)\big|\mathcal{F}_{\lfloor n_ms\rfloor}\big]ds.
\end{eqnarray*}

By Steps 2 and 3, given uniform convergence of $X^{n_m}(s)\rightarrow \varphi_k(s)$ and $\S^{n_m}(s)\rightarrow\S(s)$ for $s\in [0,N]$, and positivity of $\S^{n_m}(s)$ and $\S(s)$ for $s> 0$, in both models, we have for $0<s\leq N$,
\begin{eqnarray*}
&&
\lim_{m\uparrow\infty}E\big[d^{n_m}_k(\lceil n_ms\rceil)\big|\mathcal{F}_{\lfloor n_ms\rfloor}\big] \\
&&\  \ \ \ \ \ \ \ \ \ \ = \
= \left\{\begin{array}{rl}
p_0-q_0\frac{w(1)\varphi_1(s)}{\S(s)} &  {\rm for \ }k=1\\
\frac{q_0}{\S(s)}\big[w(k-1)\varphi_{k-1}(s) - w(k)\varphi_k(s)\big] &  {\rm for \ }k\geq 2.
\end{array}\right. 
\end{eqnarray*}

Given the pointwise bound $|d^{n_m}_k|\leq 2$, by dominated convergence, as $n_m\uparrow\infty$, we conclude
$\{\varphi_k: k\geq 1\}$ satisfies the integral equation corresponding to \eqref{ODE} with initial condition $\varphi_k(0)=c_k$ for $k\geq 1$.  

Finally, as $\{\varphi_k(t)\}$ is nonnegative by construction, by Theorem \ref{uniqueness}, we conclude it
is the unique solution to the ODEs.  Moreover, under small initial
configurations $c_k\equiv 0$, $\varphi_k(t) = a_k(p_0,q_0,s^*)t$ for $k\geq 1$.   Hence, as this is the unique limit family $\{\varphi_k: k\geq 1\}$ for each realization in a full probability set, the whole sequence $\{X^n_k: n\geq 1\}$ a.s. must converge uniformly to this solution for $k\geq 1$.
\qed

\section{Proof of Theorem \ref{uniqueness}.}
\label{ODE_uniqueness}

After several steps in the form of successive propositions, we
complete the proof Theorem \ref{uniqueness} at the end of the Section.  The condition (SUB) will be assumed throughout.

As noted, in Chapter VI in \cite{EN}, `semigroups are everywhere', and are useful in many applications.  In this vein, our strategy is to exploit the properties of a semigroup associated to the transformed ODEs \eqref{psi_ODE}.  However, our semigroup is neither compact nor `eventually' compact (see Proposition \ref{appendix_prop} in the Appendix for a proof), properties often useful in the study of population evolutions (cf. Section VI.1 in \cite{EN} and references therein).  

Also, although the transformed ODEs \eqref{psi_ODE} do not fit into the general theoretical framework of linear `Kolmogorov' differential systems, recently considered in certain host patch/parasite models \cite{thieme1}, \cite{thieme2}, \cite{barbour}, nevertheless one can use this framework to show the semigroup associated to \eqref{psi_ODE} is strongly continuous and to help estimate 
its growth rate.  

Moreover, we show the
semigroup is positive and after a shift also `quasicompact', a statement about its `essential' growth rate and one allowing certain spectral decompositions and asymptotic analysis (see \cite{Baladi} for a discussion of `quasicompactness' with respect to ergodic theory). 
Through semigroup `Perron-Frobenius' results, we specify further the evolution in terms of a finite-dimensional motion and a part corresponding to spectra in a left-half plane.  
For large initial configurations, such a decomposition is enough to capture the asymptotic growth of the ODE solutions.  

However, under small initial configurations, because of the time singularity at $t=0$, identification of the solution requires more work.
By a time-reversal argument, we show the only part of the evolution, consistent with nonnegativity of the solution and the initial condition, corresponds to motion in terms of a dominant eigenvalue-eigenvector pair, leading to the desired characterization.

\medskip
With respect to a nonnegative solution of the integral form of the nonlinear ODEs \eqref{ODE}, define
for $t\geq 0$ the nonnegative functions
$$
V(t)  =  \sum_{k\geq 1}\varphi_k(t),\ \
T(t)  =  \sum_{k\geq 1}w(k)\varphi_k(t), \ \ {\rm and \ }\
D(t)  =  \sum_{k\geq 1}k\varphi_k(t).$$
We now derive properties of the functions $V$, $T$ and $D$, representing the scaled vertices/urns, weight, and degree/size of the system respectively.  Recall $p_0,q_0>0$.
  
\begin{proposition}
\label{T_bounds}
Suppose (SUB) holds.  For $t\geq 0$, $T(\cdot)$ is continuous.  Also,
$V(t) = c+ p_0t$, $D(t)\leq \c + (p_0 + q_0)t$, and in addition there is a constant $C_0>0$ such that
$$C_0^{-1}[c+p_0t] \ \leq\  T(t)\ \leq \ C_0[\c + (p_0 + q_0)t].$$  
\end{proposition}

\begin{proof}
We first consider $D(t)$.  From the ODEs \eqref{ODE}, write for $L\geq 1$ that
\begin{eqnarray}
\label{ode_help1}
&&\sum_{k=1}^Lk\varphi_k(t) \ = \ \sum_{k=1}^Lkc_k + \sum_{k=1}^Lk\big(\varphi_k(t)
-\varphi_k(0)\big) \\ 
&&\ \ \ = \  \sum_{k=1}^L kc_k +\ p_0t + q_0\int_0^t \sum_{k=1}^{L-1} \frac{w(k)\varphi_k(u)}{T(u)}du - q_0L\int_0^t\frac{w(L)\varphi_L(u)}{T(u)}du.\nonumber
\end{eqnarray}
Hence, by (LIM), dropping the last negative term, $D(t) \leq \c + (p_0 + q_0)t$ as desired.

Considering $V(t)$ now, write
\begin{eqnarray*}
\sum_{k=1}^L\varphi_k(t)  \ = \  \sum_{k=1}^L c_k +\ p_0t - q_0\int_0^t\frac{w(L)\varphi_L(u)}{T(u)}du.\nonumber
\end{eqnarray*}
Since $\int_0^t \sum_{k\geq 1}w(k)\varphi_k(u)/T(u)du = t<\infty$, the last integral above vanishes as $L\uparrow\infty$.  Hence,
by (LIM), $V(t) = c + p_0t$.

Also, the lower and upper bounds
on $T(t)$ follow from the argument as given in Step 3 of the proof of
Theorem \ref{mainthm} in Section \ref{proofs_section}.  The constant
$C_0$ can be taken as $C_0 = \max\{[(1/2)\inf_{1\leq k\leq \hat L}w(k)]^{-1},
\w\}$ where $\hat L$ is the smallest integer satisfying $\c/(\hat L +1)\leq c/2$ and $2/(\hat L +1)\leq p_0/2$.

To show $T$ is continuous, write
$$T(t) - T(s) \ = \ \sum_{k=1}^L w(k)\big(\varphi_k(t) -\varphi_k(s)\big) + \sum_{k>L} w(k)\varphi_k(t) - \sum_{k>L} w(k)\varphi_k(s).$$
The last two terms, for large $L$, are small by the inequality $\sum_{k\geq L}w(k)\varphi_k(u) \leq [\sup_{j>L}w(j)/j]D(u)$ and (SUB).
Now, as $\{\varphi_k\}$ satisfies the integral form of the ODEs \eqref{ODE}, they are continuous functions.  Hence, one sees $T$ is also continuous. 
\end{proof}

We now analyze more carefully the time scale $t=t(s)$ and associated
system $\{\psi_k(s) = \varphi_k(t(s))\}$ mentioned
above the statement of Theorem \ref{uniqueness}.  Recall
$$\dot t(s) \ = \ T(t) \ \ \ {\rm and } \ \ \ t(0)\ =\ 1.$$
  Since $T$ is continuous by Proposition \ref{T_bounds}, a solution $t=t(s)$ exists.  Also, given the bounds on $T$ in Proposition \ref{T_bounds}, by comparison estimates, we have for $s\geq 0$ that
\begin{eqnarray}
\label{t(s)_expression}
&& t(0)e^{C^{-1}_0p_0s} +
 \frac{c}{p_0}\big[e^{C^{-1}_0p_0s} - 1\big] \\
&&\ \ \ \ \leq \  t(s) \ \leq \ t(0)e^{C_0(p_0+q_0)s} +
 \frac{\c}{(p_0+q_0)}\big[e^{C_0(p_0+q_0)s} - 1\big],\nonumber \\
&&t(0)e^{-C_0(p_0+q_0)s} +
\frac{\tilde c}{(p_0+q_0)}\big[e^{-C_0(p_0+q_0)s} - 1\big] \nonumber \\
&& \ \ \ \ \ \leq \ t(-s) \ \leq \ t(0)e^{-C^{-1}_0p_0s} +
\frac{c}{p_0}\big[e^{-C^{-1}_0p_0s} - 1\big]. \nonumber
\end{eqnarray}
In addition, as
$T(t)>0$ for $t>0$, $t(s)$ is a strictly increasing, invertible function of $s$.
Then, under
small initial configurations $c_k\equiv 0$,
as $s\downarrow -\infty$, we have $t(s) \downarrow 0$.  Under large
initial conditions $c>0$, there is an $-\infty<s_0<0$ where $t(s_0)=0$.

 The system $\psi_k(s) = \varphi_k(t(s))$ for $k\geq 1$ obeys the integral form of ODEs \eqref{psi_ODE},
with boundary conditions, under small initial configurations, $\lim_{s\downarrow -\infty}\psi_k(s) =0$ and, under large initial configurations, $\psi_k(s_0)= c_k$, for $k\geq 1$.  Also, given $\varphi_k(\cdot)\geq 0$, of course, in the corresponding time-ranges $\psi_k(\cdot)\geq 0$ for $k\geq 1$.

With $\Psi = \langle \psi_k: k\geq 1\rangle$, 
\begin{equation}
\label{psi-ODE}
\dot\Psi \  = \  A\Psi \end{equation}
where
$$A \ = \ \left(\begin{array}{rrrr}
(p_0-q_0)w(1)& p_0w(2) & p_0w(3)& \cdots\\
q_0w(1)&-q_0w(2) & 0 & \cdots\\
0&q_0w(2) &-q_0w(3) & \cdots\\
\vdots&\vdots & \vdots& \ddots \end{array}\right)$$
It will be convenient to write $A = B + K$ where
$$B \ = \ \left(\begin{array}{rrrr}
-q_0w(1)& 0 & 0& \cdots\\
q_0w(1)& -q_0w(2) & 0 & \cdots\\
0&q_0w(2) & -q_0w(3) & \cdots\\
\vdots&\vdots & \vdots& \ddots \end{array}\right)$$

$$K \ = \  \left(\begin{array}{rrrr}
p_0w(1)& p_0w(2) & p_0w(3)& \cdots\\
0&0 & 0 & \cdots\\
0&0 & 0 & \cdots\\
\vdots&\vdots & \vdots& \ddots \end{array}\right).$$

For a vector $x =\langle x_k: k\geq 1\rangle$ where $x_k\in \R$,
define the norm $\|x\| = \sum_{k\geq 1}k|x_k|$ and the Banach lattice (cf. Section VI.1b in \cite{EN})
$$\Omega \ = \ \big\{x = \langle x_k: k\geq 1\rangle: \|x\|<\infty\big\}.$$

Let $\ell_c$ be the space of compactly supported vectors, and note
$\ell_c\subset \Omega$.  The operators $A$, $B$ and $K$ are
well-defined on $\ell_c$.  Moreover, $A\ell_c \subset \ell_c$,
$B\ell_c\subset \ell_c$ and $K\ell_c \subset \ell_c$, and hence $A,B,
K$ are densely defined.  

Moreover, $K:\Omega \rightarrow \Omega$ is a bounded operator:  As $w(k)\leq \w k$ for $k\geq 1$ (SUB), the bound $\|K\| \leq p_0\w$ may be computed.  Also, since $K$ is a bounded, rank $1$ operator on
$\Omega$, $K$ is in addition compact.  

The operator $B$ with domain $\ell_c$ is closable:  Indeed, let $\{x^n\}\subset \ell_c$ so that $x^n \rightarrow 0$ and $Bx^n \rightarrow y$ in $\Omega$.  Since every row in $B$ is in $\ell_c$, we have $(Bx^n)_k \rightarrow 0$ for $k\geq 1$.  Since projections $\pi_k:\Omega \rightarrow \R$ where $\pi_k(x) = x_k$ is continuous, we have $\pi_k(Bx^n) \rightarrow \pi_k(y)= 0$ for $k\geq 1$.  Hence, $y=0$.  

As a consequence, $(A, \ell_c)$ is closable.  We will denote the
closures of $A$ and $B$ by the same names as it will not cause
confusion in what follows.

Now, we observe the ODEs associated to $B$, $\dot\zeta = B\zeta$, fall into the framework of the `Kolmogorov' differential equations considered in \cite{thieme2}.  Indeed, given $\sup_k w(k)/k\leq \w$, in the notation of \cite{thieme2}, with $\alpha_{k,k}=-q_0w(k+1)$, $\alpha_{k+1,k}=q_0w(k+1)$ for $k\geq 0$, $\alpha_{j,k} = 0$ otherwise, $\alpha^\diamond = \sup_k\sum_{j=0}^\infty \alpha_{j,k} = 0$, $c_0 = 2q_0\w$, $c_1=2q_0\w$, $\epsilon =1$ and $\omega = c_1\vee (\alpha^\diamond + c_0) = c_0$, one inspects $\sum_{j=1}^\infty j\alpha_{j,k} = q_0w(k+1)$, and the $B$-ODE system satisfies Assumptions 1, 2 in \cite{thieme2}.  We note the full statement of (SUB) is not used in this verification.

\begin{proposition}
\label{P_tbound}
Both $A$ and $B$ generate strongly continuous semigroups
$P_t, P_t^B:\Omega\rightarrow \Omega$ with bounds $\|P_t\| \leq
e^{(2q_0\w+\|K\|)t}$ and $\|P_t^B\|\leq e^{2q_0\w t}$ for $t\geq 0$ respectively.
\end{proposition}

\begin{proof}
By Theorem 2 in \cite{thieme2} there is a strongly continuous semigroup $P_t^B:\Omega \rightarrow \Omega$, generated by the part of $B$ restricted to domain $D(B) = \{x\in \Omega: \sum_k w(k)|x_k|<\infty, Bx\in \Omega\}$, with bound $\|P_t^B\|\leq e^{\omega t}$.    

Moreover, by the perturbation Theorem III.1.3 in \cite{EN}, as $K$ is
bounded, $A = B+K$ with domain $D(A)=D(B)$ generates a strongly continuous
semigroup $P_t:\Omega \rightarrow \Omega$ with bound $\|P_t\| \leq
e^{(\omega +\|K\|)t}$.
\end{proof}

Recall that a strongly continuous semigroup $P_t^E:\Omega \rightarrow \Omega$ is
positive if $(P^E_t x)_k\geq 0$ for $k\geq 1$ when $x\in
\Omega$ and $x_k\geq 0$ for all $k\geq 1$
(cf. Section VI.1b in \cite{EN}).  

Below, in Proposition \ref{positive_prop}, we show the semigroups generated by $B$ and $A$ are positive.  
In passing, we remark in fact $P_t$, although not $P_t^B$, is irreducible, that is $[(\lambda I -A)^{-1}x]_k>0$ for $k\geq 1$ when $x_k\geq 0$ for all $k\geq 1$ but $x\neq 0$ (cf. Section VI.1b in \cite{EN}).    Indeed, from the ODEs \eqref{psi-ODE} and a calculation left to the reader, $(P_sx)_k = \psi_k(s)>0$ for $k\geq 1$ and $s>0$.  
Then, $(\lambda I - A)^{-1}x = \int_0^\infty e^{-\lambda s}P_sx ds$ is composed of positive entries.  We will not need this stronger result in the following.

\begin{proposition}
\label{positive_prop}
The semigroups $P_t$ and $P_t^B$ are both positive.
\end{proposition}

\begin{proof}
We show $P_t$ is positive; the same argument also proves $P_t^B$ is positve.  Since $\ell_c$ is a core of $A$, we can calculate
$P_t x$ for $x\in \ell_c$ such that $x_k\geq 0$ for $k\geq 1$ by the equation $(d/dt)P_t x = A P_t x$, in
other words ODEs \eqref{psi-ODE}, and initial condition
$P_0x = x$ (cf. Lemma II.1.3 in \cite{EN}).  As $q_0>0$, all off-diagonal entries in $A$ are nonnegative.  Hence, inspection of these ODEs
reveals that $P_tx \geq 0$.  Now, for $x\in \Omega$ and $x_k\geq 0$ for
$k\geq 1$, take $x_n \in \ell_c$ so that
$(x_n)_k\geq 0$ for $k\geq 1$  and $x_n \rightarrow x$ in $\Omega$.
Since, for fixed $t\geq 0$, $P_t$ is bounded (cf. Proposition \ref{P_tbound}),
$P_tx_n \rightarrow P_t x$ in $\Omega$.  Hence, $P_tx \geq 0$.
\end{proof}

The growth rate $w_0(E)$ of a semigroup $P_t^E$ is
$w_0(E) = \lim_{t\uparrow\infty} t^{-1} \log \|P_t^E\|$ (cf. Proposition IV.2.2 in \cite{EN}).  Also, the essential growth rate $w_{\rm ess}(E)$ of $P_t^E$ is 
$w_{\rm ess}(E)  =  \lim_{t\uparrow\infty}t^{-1}\log \|P_t^E\|_{\rm ess}$ where
$\|P_t^E\|_{\rm ess} = \inf\{ \|P_t^E - M\|: M \ {\rm \ compact}\}$ (cf. Proposition IV.2.10 in \cite{EN}).  In particular, inputting $M\equiv 0$, we obtain 
$$w_{\rm ess}(E) \ \leq \ w_0(E).$$

\begin{proposition}
\label{w_0(B)_bound}
 We have that $w_0(B)\leq 0$. \end{proposition}

\begin{proof}  Again, since the $B$-ODEs satisfy Assumptions 1,2 in \cite{thieme2}, by Theorem 4 in \cite{thieme2}, we have 
$$w_0(B)\ \leq \ \alpha^\diamond \vee \limsup_{k\rightarrow\infty} \sum_{j=1}^\infty j\alpha_{j,k}/k \ =\  0,$$
recalling $\alpha^\diamond =0$ and $\limsup_{k\uparrow\infty}\sum_{j=1}^\infty j\alpha_{j,k}/k = \lim_{k\uparrow\infty}q_0w(k+1)/k = 0$.
\end{proof}

We now show for all small $\varepsilon>0$ that $e^{-\varepsilon t}P_t$ is a quasi-compact semigroup, that is the essential growth rate $w_{\rm ess}(A-\varepsilon I)<0$.  This is one characterization of being `quasicompact' (cf. Proposition V.3.5 in \cite{EN}).  Such semigroups have nice representations which we will leverage later on.

\begin{proposition}
\label{quasi_compact}
For all small $\varepsilon>0$, the semigroup $e^{-\varepsilon t}P_t$ is quasi-compact.
\end{proposition}

\begin{proof}
We will show that $e^{-\varepsilon t}P^B_t$, the semigroup generated
by $B-\varepsilon I$, is quasi-compact.  Then, by the perturbation
result Proposition V.3.6 in \cite{EN}, as $K$ is a compact operator, $e^{-\varepsilon t}P_t$ the semigroup generated by $A-\varepsilon I = B+K-\varepsilon I$ is also quasi-compact.

As stated in Proposition V.3.5 in \cite{EN}, for a strongly continuous semigroup,
quasi-compactness is equivalent to the essential growth rate of the
semigroup being strictly negative.  We will apply this characterization to $B- \varepsilon I$.   
Since $w_0(B)\leq 0$ by Proposition \ref{w_0(B)_bound}, we have
$w_{\rm ess}(B-\varepsilon I)  \leq  w_0(B-\varepsilon I)  =  w_0(B) - \varepsilon   < 0.$
\end{proof}

 Now, by the quasi-compact semigroup representation Theorem V.3.7 in
 \cite{EN} applied to $e^{-\varepsilon t}P_t$ for $\varepsilon>0$, there are only a finite number $m$ of spectral values $z$, if any, of $A-\varepsilon$, and each of these is a pole of the resolvent $R(\cdot, A-\varepsilon)$ with finite algebraic multiplicity.
Moreover, when $m\geq 1$, we may write for $t\geq 0$ that
\begin{equation}
\label{quasi_compact_eqn}
e^{-\varepsilon t}P_t \ = \ \sum_{r=1}^m U_r(t) + R(t).\end{equation}
Here, with respect to the $r$th pole $\lambda_r$ with multiplicity $k_r$
and spectral projection $Q_r$ (cf. Proposition IV.1.16 in \cite{EN}),
$$U_r(t) \ = \ e^{\lambda_r t} \sum_{j=0}^{k_r -1}
\frac{t^j}{j!}(A-(\varepsilon + \lambda_r)I)^j Q_r.$$
Also, Theorem V.3.7 in \cite{EN} states $\|R(t)\| \leq M e^{-\beta t}$ for some $\beta>0$ and $M\geq 1$.

In effect, $e^{-\varepsilon t}P_t$ acts as a finite-dimensional operator on ${\rm Range}(Q_r)$ and leaves it invariant for $1\leq r\leq m$.  In particular, $e^{-\varepsilon t}P_t$ and $\{Q_r\}$ commute and
\begin{equation}
\label{R_t_rep}
e^{-\varepsilon t}P_t Q_r = U_r(t) \ \ {\rm and \ \ } R_t = e^{-\varepsilon t}P_t\Big[ I - \sum_{r=1}^m Q_r\Big ].
\end{equation}

Let now $\sigma(E)$ be the spectrum of a generator $E$ on $\Omega$.  The largest real part of the spectrum is denoted
$s(E) 
= \sup\{{\rm Re}(\lambda): \lambda \in
\sigma(E)\}$.  

  To make use of this representation, 
we now examine the spectrum of $A$ in a right half plane. A goal in the next propositions is to show that $s(A)$ is positive and a simple eigenvalue.  Also, we derive the form of its eigenvector.

\begin{proposition}
\label{s_positive}
The generator $A$ has only one real eigenvalue in
the strict right
half-plane $\{z\in \C: {\rm Re}(z)>0\}$, and it has an eigenvector with all positive entries.  As a consequence, $s(A)>0$.
\end{proposition}

\begin{proof}
We solve $Ax = \lambda x$ for $\lambda >0$.  We have
\begin{eqnarray*}
\lambda x_1 &=&(p_0-q_0)x_1 + p_0\sum_{k\geq 2} w(k)x_k\\
\lambda x_k &=& q_0\big\{w(k-1)x_{k-1} - w(k)x_{k}\big\} \ \ \  {\rm for \ } k\geq 2.
\end{eqnarray*}
This gives, for $k\geq 2$,
\begin{equation}
\label{dominant_eigenvector}
x_k \ = \ x_1\prod_{r=2}^k\frac{q_0w(r-1)}{\lambda +
  q_0w(r)},\end{equation}
the same equations for $a_k(p_0,q_0,\lambda)$ (cf. \eqref{discrete_ODE}).  

In particular, by (SUB), a calculation shows that $x\in \Omega$, and
$$\sum_{k\geq 2} w(k)x_k \ = \ x_1w(1)\sum_{k\geq 2}\prod_{r=2}^k
\frac{q_0w(r)}{\lambda + q_0w(r)}$$
 converges for $\lambda>0$.
Hence, plugging
into the equation involving $x_1$ above,
\begin{equation}
\label{lambda_eqn}
\lambda  \ = \ p_0-q_0 + p_0w(1)\sum_{k\geq 2}\prod_{r=2}^k
\frac{q_0w(r)}{\lambda + q_0w(r)}.\end{equation}
 The left side of the equation \eqref{lambda_eqn} is strictly increasing in $\lambda$, whereas the
 right-side is strictly decreasing in $\lambda$.  Also, the right-side
 of \eqref{lambda_eqn} diverges to infinity as $\lambda\downarrow 0$.  We
 conclude therefore there is exactly one $\lambda>0$ which satisfies
 \eqref{lambda_eqn}.  This $\lambda$ is the desired unique real
 eigenvalue, with positive eigenvector $x$ when $x_1>0$.
\end{proof}

\begin{proposition}
\label{max_eigenvalue}
For $0\leq \varepsilon< s(A)$, $s(A-\varepsilon I)>0$ 
is the only real eigenvalue of $A - \varepsilon I$ in the strict right-half plane $\{z\in \C: {\rm Re}(z)> 0\}$.  All other
eigenvalues $\lambda$ of $A-\varepsilon I$, if they exist, satisfy
${\rm Re}(\lambda)< s(A-\varepsilon I)$.
\end{proposition}

\begin{proof}
First, for $\varepsilon>0$, as $e^{-\varepsilon t}P_t$ is quasi-compact 
(Proposition \ref{quasi_compact}), as noted above there are only a finite number of spectral values of
$A-\varepsilon I$ in the right half-plane $\{z\in \C: {\rm Re}(z)\geq 0\}$, and these are all eigenvalues.
In particular, there are only a finite number of spectral values/eigenvalues of $A$ in the half-plane $\{z\in \C: {\rm Re}(z)\geq \varepsilon\}$.

Then, as $s(A)>0$ (Proposition \ref{s_positive}), and by positivity of $P_t$ (Proposition \ref{positive_prop})
and the `Perron-Frobenius' type Theorem VI.1.10 in \cite{EN}, $s(A)$ is
an eigenvalue of $A$ and, by Proposition \ref{s_positive}, the only real one in the strict right-half plane.  Moreover, with $\varepsilon = s(A)/2$, as
there are only a finite number of eigenvalues $z$ of $A$ with real part ${\rm Re}(z)\geq s(A)/2$, 
by another `Perron-Frobenius' type Theorem VI.1.12(i) in \cite{EN}, the boundary
spectrum of $A$ must be a singleton.  Hence, any other eigenvalue
$z$ of $A$
satisfies ${\rm Re}(z)<s(A)$.

Then, for all $0\leq \varepsilon<s(A)$, $s(A-\varepsilon I) = s(A) -
\varepsilon$ is the only real eigenvalue of $A-\varepsilon I$ in the
strict right half-plane, and all other eigenvalues have real part strictly less than $s(A-\varepsilon I)$.  
\end{proof}

Define now the dual space 
$$\Omega' \ = \ \big\{z: {\rm \ There \ exists \ }C {\rm \ such \ that \ } |z_k|\leq Ck \ {\rm for \ all \ }k\geq 1\big\}$$
and $\|z\|_{\Omega'}$ is the smallest such constant $C$.  
It will be helpful now to find an eigenvector of 
$$A^* \ = \ \left(\begin{array}{rrrrr}
(p_0-q_0)w(1)&q_0w(1)&0&0&\cdots\\
p_0w(2)&-q_0w(2)& q_0w(2)&0&\cdots\\
p_0w(3)&0&-q_0w(3)&q_0w(3)&\cdots\\
\vdots&\vdots&\vdots&\vdots&\cdots\end{array}\right)$$
 with positive
entries.

\begin{proposition}
\label{adjoint_eigenvalue} 
There exists an eigenvector $x^*\in \Omega'$ of $A^*$, with all entries
positive, corresponding to a real eigenvalue $\lambda^*>0$.  Moreover, $\lambda^*$ can be taken
$\lambda^*= s^*$ where we recall $s^*$ solves $1= F_{p_0,q_0}(s^*)$
(cf. \eqref{s*defn}).

\end{proposition}

\begin{proof}
For a possible eigenpair $x^*, \lambda^*$, we obtain equations
\begin{eqnarray}
\label{adjoint_eqns}
x_1^* &=& \frac{q_0w(1)x_2^*}{\lambda^* + q_0w(1)} +
\frac{p_0w(1)x_1^*}{\lambda^* + q_0w(1)}\\
x_k^*&=& \frac{q_0w(k)x_{k+1}^*}{\lambda^* + q_0w(k)} +
\frac{p_0w(k)x_1^*}{\lambda^* + q_0w(k)} \ \ \ {\rm for \ }k\geq
2.\nonumber
\end{eqnarray}

Note, by (SUB), the sum
$$\sum_{k\geq 2} \frac{p_0w(k)}{\lambda^* + q_0w(k)}\prod_{r=1}^{k-1} \frac{q_0w(r)}{\lambda^* +
  q_0w(r)}$$
converges for each $\lambda^*>0$.  
Also, consider the equation
\begin{equation}
1 \ = \ \sum_{k\geq 2} \frac{p_0w(k)}{\lambda^* + q_0w(k)}\prod_{r=1}^{k-1} \frac{q_0w(r)}{\lambda^* +
  q_0w(r)} + \frac{p_0w(1)}{\lambda^* + q_0w(1)},
\label{adjoint_equality}\end{equation}
which is the same as $1=F_{p_0,q_0}(\lambda^*)$ and identifies, as concluded in
\eqref{s*defn}, $\lambda^* = s^*$.

Iterating \eqref{adjoint_eqns}, we may solve
\begin{eqnarray*}
x_1^* & = & \lim_{N\uparrow\infty} x_{N+1}^*\prod_{r=1}^N 
\frac{q_0w(k)}{\lambda + q_0w(k)} \\
&&\ \ \ + x_1^*\sum_{k\geq 2} \frac{p_0w(k)}{\lambda + q_0w(k)}\prod_{r=1}^{k-1} \frac{q_0w(r)}{\lambda +
  q_0w(r)} + x_1^*\frac{p_0w(1)}{\lambda + q_0w(1)}.\end{eqnarray*}
With $\lambda^*=s^*$, necessarily, noting \eqref{adjoint_equality}, $\lim_{N\uparrow\infty} x_{N+1}^*\prod_{r=1}^N 
\frac{q_0w(k)}{\lambda^* + q_0w(k)}=0$.   

In this case, for $j\geq 2$, with convention $\prod_{r=j+1}^j\cdot = 1$,
\begin{eqnarray*}
x_j^* & = &  x_1^*\sum_{k\geq j} \frac{p_0w(k+1)}{\lambda + q_0w(k+1)}\prod_{r=j}^k \frac{q_0w(r)}{\lambda +
  q_0w(r)} + x_1^*\frac{p_0w(j)}{\lambda + q_0w(j)}\\
  &= & \frac{x_1^*w(j)}{\lambda + q_0w(j)}\Big[ \sum_{k\geq j} \frac{p_0w(k+1)}{\lambda + q_0w(k+1)}\prod_{r=j+1}^k \frac{q_0w(r)}{\lambda +
  q_0w(r)} + p_0\Big].\end{eqnarray*}
  Again, by (SUB), one sees that
  $|x_j^*|\leq Cj$
  for a uniform constant $C$ for $j\geq 1$.  In particular, the eigenvector $x^*\in \Omega'$ and if $x_1^*>0$, $x_j^*>0$ for $j\geq 2$.
\end{proof}

\begin{proposition}
\label{s(A)_simple}
The eigenvalue $s(A)$ of $A$ is simple and moreover $\lambda^* = s(A)$.
Also, $x^* \perp w$ for any (generalized) eigenvector $w$ of $A$ other than the one with eigenvalue $s(A)$.
\end{proposition}

\begin{proof}
Consider the eigenvector $x^*$ with eigenvalue $\lambda^*$ in Proposition \ref{adjoint_eigenvalue} consisting of all positive entries when say $x_1^* =1$.  Then, with respect to the positive eigenvector $x$ with eigenvalue $s(A)$ of $A$ in Proposition
\ref{max_eigenvalue}, we note $\langle x^*, Ax\rangle = s(A)\langle x^*, x\rangle
= \lambda^*\langle x^*, x\rangle$.  Since $\langle x^*, x\rangle>0$,
$\lambda^* = s(A)$.  

Moreover, suppose there exists a generalized eigenvector $w$ where $(A-s(A)I)w = cx$ for some $c\neq 0$.  Then,
$\lambda^*\langle x^*, w\rangle = \langle x^*, Aw\rangle = \langle x^*, s(A)w\rangle + c\langle x^*, x\rangle$.  Since $\lambda^* = s(A)$ and $\langle x^*,x\rangle>0$, we must have $c=0$ which is a contradiction.  Hence, $s(A)$ is a simple eigenvalue of $A$.  

Finally, for any eigenvector $w$ of $A$ with eigenvalue $\lambda_w\neq s(A)$, $s(A) \langle x^*, w\rangle = \langle x^*, Aw\rangle = \lambda_w \langle x^*, w\rangle$.  Since $\lambda_w \neq s(A)$,  we have $x^*\perp w$.  If $w'$ is a generalized eigenvector corresponding to eigenvalue $\lambda_w$, $(A-\lambda_w)^kw' =  cw$ for some power $k$ and constant $c$.  Then, $(s(A) - \lambda_w)^k\langle x^*, w'\rangle = \langle (A^*-\lambda_w)^k x^*, w'\rangle = c\langle x^*, w\rangle = 0$.  Again, as $\lambda_w \neq s(A)$, $x^*\perp w'$.   
\end{proof}

Consider now, under small initial configurations, the global trajectory $\psi(s) =\langle \psi_k(s): k\geq
1\rangle$ satisfying the ODEs \eqref{psi-ODE} such that $\psi(0) =
\phi(1)$ and $\lim_{s\downarrow -\infty}\psi(s) = 0$.  To characterize $\psi(s)$, we will need the following estimate.

\begin{lemma}
\label{time_change_bound} 
Under small initial configurations, for
$s<0$,
$$Y(s) \ :=\  \|\psi(s)\|\  \leq \ (p_0+q_0)e^{C_0^{-1} s}$$ 
where $C_0$ is the constant in Lemma \ref{T_bounds}.
\end{lemma}

\begin{proof}
From \eqref{t(s)_expression}, applied to small initial configurations ($c=0$), noting $t(0)=1$,
we have $t(s)\leq e^{C_0^{-1}s}$.  Also, note that $D(u) \leq (p_0+q_0) u$ for $u\geq 0$.  Then,
$Y(s) = \sum_k k\psi_k(s) = \sum_k k\varphi_k(t(s)) = D(t(s)) \leq (p_0 + q_0)t(s) \leq (p_0+q_0)e^{C_0^{-1}s}$.
\end{proof}

\begin{proposition}
\label{psi_characterization}
With respect to small initial
configurations, we identify $\psi(s) = e^{s(A)
  s}\psi(0)$ for all $s\in \R$ where $\psi(0)$ is an eigenvector of $A$ with eigenvalue $s(A)$.
\end{proposition}

\begin{proof}  The argument proceeds in steps.
\medskip

\noindent {\it Step 1.}
Let $0<\varepsilon<
C_0^{-1}/2$ where $C_0$ is the constant in Proposition \ref{T_bounds}.  Recall the quasi-compact representation of $e^{-\varepsilon t}P_t$ in \eqref{quasi_compact_eqn}.
For $u\in \R$, define $\xi(u)$ by the equation
\begin{equation}
\label{help0}
e^{-\varepsilon u}\psi(u) \ = \ e^{-\varepsilon u}\sum_{r=1}^m Q_r\psi(u) + \xi(u).\end{equation} 
Then, for $t\geq 0$ and $s\in \R$, on the one hand, 
\begin{equation}
\label{help1}
e^{-\varepsilon(s+t)}\psi(s+t) \ =\  e^{-\varepsilon(s+t)}\sum_{r=1}^m Q_r\psi(s+t) + \xi(s+t).\end{equation}
On the other hand, as $R_t = e^{-\varepsilon t}P_t[I - \sum_{r=1}^m Q_r]$ (cf. \eqref{R_t_rep}),
\begin{eqnarray*}
e^{-\varepsilon(s+t)}\psi(s+t)  &=& e^{-\varepsilon t}P_t\big(e^{-\varepsilon s}\psi(s)\big)\nonumber\\
&=& e^{-\varepsilon t}P_t\big[\sum_{r=1}^m Q_r\big]\big(e^{-\varepsilon s}\psi(s)\big) + R_t\big(e^{-\varepsilon s}\psi(s)\big).\end{eqnarray*}
Since $R_t\big(e^{-\varepsilon s}\psi(s)\big) = R_t\xi(s)$, and $e^{-\varepsilon t}P_t$ and $\{Q_r\}$ commute,
\begin{equation}
\label{help2}
e^{-\varepsilon (s+t)}\psi(s+t) \ = \ e^{-\varepsilon (s+t)}\sum_{r=1}^m Q_r \psi(s+t) + R_t\xi(s).\end{equation}
 Hence, combining \eqref{help1} and \eqref{help2}, we have $\xi(s+t) = R_t\xi(s)$, and with the bound on $R_t$ after \eqref{quasi_compact_eqn},
$$ \|\xi(s+t)\|
\ \leq \
\|R(t)\|\|\xi(s)\| \ \leq \ Me^{-\beta t}\|\xi(s)\|.$$

\medskip
\noindent {\it Step 2.}
We now argue that $\xi(u) = 0$ for all $u\in \R$. 
First, for $s<0$, from Lemma \ref{time_change_bound} and $\varepsilon < C_0^{-1}/2$, 
\begin{equation}
\label{bound1}
\|e^{-\varepsilon s}\psi(s)\| \ = \ e^{-\varepsilon s}Y(s) \ \leq \ (p_0 + q_0)e^{(C_0^{-1}-\varepsilon) s}
\ \leq \ (p_0 + q_0)e^{(C_0^{-1}/2) s}.
\end{equation}

Second, from its finite-dimensional
 form, the operator $e^{-\varepsilon t}P_t|_{{\rm Range}(Q_r)}$ is
 invertible for $t\geq 0$.  Denote the inverse on the range of $Q_r$ as
 $$e^{\varepsilon t}P_{-t}|_{{\rm Range}(Q_r)} \ = \ e^{-\lambda_r t}\sum_{j=0}^{k_r-1} ((-t)^j/j!)(A-(\varepsilon + \lambda_r)I)^jQ_r \ =: \  U_r(-t)$$
where $U_r$ is extended to $\R_-$. 
Then, for $s<0$, we have
$$e^{\varepsilon s}P_{-s} \sum_{r=1}^m Q_r \big(e^{-\varepsilon s}\psi(s)\big) \ = \ \sum_{r=1}^m Q_r\big(P_{-s}\psi(s)\big) \ = \ \sum_{r=1}^m Q_r \psi(0).$$
Hence, after inverting,
\begin{equation}
\label{help4}
e^{-\varepsilon s}\sum_{r=1}^m Q_r\psi(s) \ = \ \sum_{r=1}^m U_r(s)\psi(0).\end{equation}

Third, with respect to a constant $C=C(\{\lambda_r\}, \{k_r\}, \varphi(1))$, for $s<0$, from \eqref{help0} and \eqref{help4}, and bound \eqref{bound1} and $\lambda_r\geq 0$ for $1\leq r\leq m$,
\begin{eqnarray*}
\|\xi(s)\| & \leq  & \|e^{-\varepsilon s} \psi(s)\| +
\big\|\sum_{r=1}^m U_r(s)\psi(0)\big\| \\
& \leq & (p_0+q_0) + C|s|^{\max_{1\leq r\leq m} {k_r -1}} 
\end{eqnarray*}

As a consequence, for fixed $u = s+t$ where $s<0$ and $t>0$, as $t\uparrow\infty$,
we have
$$\|\xi(u)\| \ = \ \|\xi(s+t)\| \ \leq \ Me^{-\beta t}\big[(p_0+q_0)+ C|u-t|^{\max
  k_r -1}\big] \ \rightarrow \ 0.$$
Therefore, in equation \eqref{help0}, $e^{-\varepsilon s}\psi(u) = \sum_{r=1}^m U_r(u)\psi(0)$ for all $u\in \R$.  
\medskip

\noindent {\it Step 3.}
Recall $\psi(\cdot)$ is assumed nonnegative.  We now show that $\psi(u) = e^{s(A)
  u}\psi(0)$ for all $u\in \R$ where $\lambda=s(A)$ is the simple eigenvalue of $A$
with largest real part.  We will also conclude $\psi(0)$ is an
eigenvector corresponding to $s(A)$.

Indeed, the eigenvalue $\lambda_r$ with largest real part is of form $s(A) -
\varepsilon>0$ with a corresponding eigenvector $x$ with all positive
entries (cf. Proposition \ref{s_positive}).  Recall $x^*$ the positive eigenvector of $A^*$ with
eigenvalue $\lambda^*=s(A)$ and that all (generalized) eigenvectors $x^r$ of $A-\varepsilon I$ corresponding to
$\lambda_r \neq s(A) -\varepsilon$ are orthogonal to $x^*$ (cf. Propositions \ref{max_eigenvalue}, \ref{adjoint_eigenvalue}, \ref{s(A)_simple}).

Let $\{\lambda_r: r\in I_\alpha\}$ be those eigenvalues with the same real part ${\rm Re}(\lambda_r)=\alpha$, and $I_\alpha$ the corresponding index set.  For $0\leq j\leq \max_{1\leq r\leq m} k_r-1$, consider the sum
\begin{eqnarray*}
A(\alpha,j,s) & :=& \sum_{r\in I_\alpha} e^{\lambda_r s}(s^j/j!)(A-(\varepsilon + \lambda_r)I)^jQ_r\psi(0)\\
&=& e^{s\alpha}\frac{s^j}{j!}\sum_{r\in I_\alpha} e^{is{\rm Im}(\lambda_r)} ((A-(\varepsilon + \lambda_r)I)^jQ_r\psi(0).\end{eqnarray*}
There are a finite number of nontrivial sums indexed by $\alpha, j$.
Let $\bar{\alpha}$ be the minimum real part of the eigenvalues $\{\lambda_r\}$ and suppose $\bar{\alpha}\neq s(A)$, the largest real part.  Let $\hat{j}$ be the maximum of $k_r-1$ among the eigenvalues $\lambda_r$ with real part $\bar{\alpha}$.  

 Suppose $A(\bar{\alpha},\hat{j},s)\neq 0$ for some $s\in \R$.  We claim we can find integers $\{k_{r,\ell}: r\in I_\alpha\}$ and $n_\ell\geq 1$ where $\lim_{\ell\uparrow\infty}n_\ell=\infty$ and $\max_{r\in I_r} |n_\ell {\rm Im}(\lambda_r)/(2\pi) - k_{r,\ell}|\leq n_\ell^{-1/|I_\alpha|}$ for $\ell\geq 1$:  Indeed, if $\{{\rm Im}(\lambda_r): r\in I_\alpha\}$ are all rational, this is the case; if one of $\{{\rm Im}(\lambda_r): r\in I_\alpha\}$ is irrational, then Dirichlet's simulataneous Diophantine approximation theorem, Corollary II.1B in \cite{Schmidt}, implies the claim.  

Then, at times $u_\ell=s- n_\ell$ for $\ell\geq 1$, the sum $(e^{u_\ell \bar\alpha}{u_\ell}^{\hat j}/\hat{j}!)^{-1}A(\bar{\alpha}, \hat{j}, u_\ell)$ well approximates $(e^{s\bar\alpha}s^{\hat j}\hat{j}!)^{-1}A(\bar{\alpha}, \hat{j}, s)$ in $\Omega$, and the absolute value $|A(\bar \alpha, \hat j, u_\ell)|$
dominates the magnitudes of all the other sums $A(\alpha,j,u_\ell)$ for $(\alpha,j) \neq (\bar{\alpha},\hat{j})$ as $\ell\uparrow\infty$.

Therefore, 
$$ e^{\varepsilon u_\ell}\psi_k(u_\ell) \ = \ 
\Big[\sum_{r=1}^m U_r(u_\ell) \psi(0)\Big]_k \ \sim \ A(\bar{\alpha}, \hat j, u_\ell)_k$$
for $|u_\ell|$ large with respect to components $k$ of $A(\bar{\alpha}, \hat j, u_\ell)$ which are nonzero.  Given $x^*\perp x^r$ for any generalized eigenvector $x_r$ of $\lambda_r$, $r\in I_{\bar{\alpha}}$, and $e^{\varepsilon u_\ell}\psi(u_\ell)$ is real, there must be a component of $A(\bar{\alpha},\hat j, u_\ell)$ which is also real and strictly negative.  This contradicts the nonnegativity of $e^{\varepsilon u_\ell}\psi(u_\ell)$.
Therefore, $A(\bar{\alpha},\hat{j},s) = 0$ for $s\in \R$.  

Similarly, considering the remaining finite number of sums $A(\alpha,j,u)$, strictly ordered according to their growth as $u\downarrow -\infty$, we conclude $A(\alpha, j,s)=0$ when $\alpha<s(A)$ for $s\in\R$.

Then, for $u\in \R$, $\sum_{r: \lambda_r\neq s(A)}U_r(u)x_0=0$ and
 so $\psi(s) = e^{s(A) u}Q\psi(0)$ where $Q$ is projection onto the eigenvector $x$ of $\lambda = s(A)$ (cf. \eqref{dominant_eigenvector}).  Finally, $\psi(0)$ is also a corresponding eigenvector since
 $\psi(0) = Q\psi(0)$.
\end{proof}

We now identify, under small initial configurations, the `time-change' $t=t(s)$' given in the
beginning of the Section. 

\begin{lemma}
\label{time-change}
With respect to small initial configurations, we have $t(u) = e^{s(A) u}$ for $u\in \R$ and $T(t) = s(A) t$
for $t\geq 0$.
\end{lemma}

\begin{proof}
Since $\psi(u) = e^{s(A) u}\psi(0)$ from Proposition \ref{psi_characterization}, and
$\psi(u) = \varphi(t(u))$, we have from Lemma \ref{T_bounds} that
$$p_0 t(u) \ = \ \sum_k \varphi_k(t(u)) \ = \ e^{s(A) u}\sum_k \psi_k(0).$$
 Since $t(0)=1$,  we have $\sum_k
\psi_k(0) = p_0$.  This shows $t(u) = e^{s(A) u}$ for $u\in \R$.
Next, as $s(A) t(u) = \dot t(u) = T(t(u))$, and $t=t(u)$ is
onto $\R$, $T(t) = s(A) t$ for $t\geq 0$.
\end{proof}

  \noindent
{\bf Proof of Theorem \ref{uniqueness}.} First, consider large initial
configurations and recall the time $s_0$ defined after \eqref{t(s)_expression} so that $t(s_0)=0$, and $\psi(s_0)_k = c_k$ for $k\geq 1$.
Then, $\psi(s+s_0) = P_s\psi(s_0)$ for $s\geq 0$ and 
$\sum_{k\geq 1}\psi_k(s+s_0) = \sum_{k\geq 1}\varphi(t(s+s_0)) =
p_0t(s+s_0) + c$.  In particular, $t(\cdot)$ is uniquely specified in terms of $\{\psi_k(\cdot)\}$ and $c$.  Hence, $\{\varphi_k(u) = \psi_k(t^{-1}(u)): u\geq 0\}$ is uniquely determined.

Moreover, for $\varepsilon>0$ small and $s>0$, as $e^{-\varepsilon
  s}P_s$ satisfies representation \eqref{quasi_compact_eqn}, and the
dominant eigenvalue $s(A) - \varepsilon>0$ is simple (Proposition \ref{s(A)_simple}), we have $e^{-s(A)
  s}\psi(s) = e^{-(s(A) - \varepsilon)s}e^{-\varepsilon s}P_s
\psi(0)$ converges in $\Omega$ to an eigenvector $v$ with eigenvalue $s(A)$ of $A$ as
$s\uparrow\infty$ (cf. discussion before Corollary V.3.3 in \cite{EN}).  
Then,
$$p_0 + \frac{c}{t(s)} \ = \ \frac{1}{t(s)}\sum_{k\geq 1}\varphi_k(t(s)) \ = \ \frac{e^{s(A) s}}{t(s)}\cdot e^{-s(A) s}\sum_{k\geq 1}\psi_k(s).$$
Taking $s\uparrow\infty$, as $t(s)\uparrow\infty$ from \eqref{t(s)_expression}, we conclude $p_0 = z\sum_{k\geq 1}v_k$
where $z = \lim_{s\uparrow\infty}e^{s(A) s}/t(s)$, which necessarily converges.
By the eigenvector
formula \eqref{dominant_eigenvector} which $\{a_k(p_0,q_0,s^*)\}$ satisfies, fact $s(A) = s^*$ (Proposition \ref{s(A)_simple}), and equality $\sum_{k\geq 1}a_k(p_0,q_0,s^*)=p_0$ (cf. \eqref{sums}), we identify $zv_k = a_k(p_0,q_0,s^*)$ for $k\geq 1$. 
Hence, for $k\geq 1$,
$$\varphi_k(s)/s \ = \ [e^{s(A) t^{-1}(s)}/s][e^{-s(A) t^{-1}(s)}\psi_k(t^{-1}(s))]\ \rightarrow \ zv_k \ = \ a_k(p_0,q_0,s^*).$$

Now, consider small initial configurations.  By Lemma \ref{time-change}, $t(s) = e^{s(A)s}$ is identified
and therefore $\{\varphi_k(u) = \psi_k(t^{-1}(u))\}$ is as well uniquely found.
 However, since $\{a_k(p_0,q_0,s^*)\}$ satisfies
 \eqref{discrete_ODE}, we conclude $\{a_k(p_0,q_0, s^*)t\}$ solves ODEs \eqref{ODE} with $c=0$.  Hence, in this case, $\varphi_k(t) = a_k(p_0,q_0,s^*)t$ for $k\geq 1$.  
\qed
  
\newpage

%\newpage

\section{Appendix:  $d_k(j+1)$ in the graph model}

As mentioned, formation of loops need to be considered.  For $k\geq 3$,
$$d_k(j+1) \ = \ \left\{\begin{array}{rl}
2& \ \ {\rm with \ prob. \ } (1-p)\big[\frac{w(k-1)Z_{k-1}(j)}{S(j)}\big]^2\\
&\ \ \ \ \ \ \ \  - (1-p)\frac{[w(k-1)]^2Z_{k-1}(j)}{[S(j)]^2}\\
1& \ \ {\rm with \ prob. \ } p\frac{w(k-1)Z_{k-1}(j)}{S(j)}\\
&\ \ \ \ \ \ \ \  + (1-p)\frac{[w(k-2)]^2Z_{k-2}(j)}{[S(j)]^2}\\
&\ \ \ \ \ \ \ \   + 2(1-p)\frac{w(k-1)Z_{k-1}(j)}{S(j)}\big[1-\frac{w(k-1)Z_{k-1}(j)}{S(j)} - \frac{w(k)Z_{k}(j)}{S(j)}\big]\\
0& \ \ {\rm with \ prob. \ } p\big[1-\frac{w(k-1)Z_{k-1}(j)}{S(j)} - \frac{w(k)Z_k(j)}{S(j)}\big] \\
&\ \ \ \ \ \ \ \  + (1-p)\frac{[w(k-1)]^2Z_{k-1}(j)}{[S(j)]^2}\\
&\ \ \ \ \ \ \ \  + 2(1-p)\frac{w(k-1)Z_{k-1}(j)}{S(j)}\frac{w(k)Z_k(j)}{S(j)}\\
&\ \ \ \ \ \ \ \  + (1-p)\big[1-\frac{w(k-1)Z_{k-1}(j)}{S(j)} - \frac{w(k)Z_k(j)}{S(j)}\big]^2\\
&\ \ \ \ \ \ \ \  - (1-p)\frac{[w(k-2)]^2Z_{k-2}(j)}{[S(j)]^2}\\
-1& \ \ {\rm with \ prob. \ } p\frac{w(k)Z_k(j)}{S(j)} \\
&\ \ \ \ \ \ \ \  + (1-p)\frac{[w(k)]^2Z_k(j)}{[S(j)]^2}\\
&\ \ \ \ \ \ \ \  + 2(1-p)\frac{w(k)Z_k(j)}{S(j)}\big[1-\frac{w(k-1)Z_{k-1}(j)}{S(j)} - \frac{w(k)Z_k(j)}{S(j)}\big]\\
-2& \ \ {\rm with \ prob. \ } (1-p)\big[\frac{w(k)Z_k(j)}{S(j)}\big]^2\\
&\ \ \ \ \ \ \ \  - (1-p)\frac{[w(k)]^2Z_k(j)}{[S(j)]^2}.
\end{array}\right. $$  
$$d_1(j+1) \ = \ \left\{\begin{array}{rl}
1& \ \ {\rm with \ prob. \ } p\big[1 - \frac{w(1)Z_1(j)}{S(j)}\big]\\
0& \ \ {\rm with \ prob. \ } p\frac{w(1)Z_1(j)}{S(j)} + (1-p)\big[1-\frac{w(1)Z_1(j)}{S(j)}\big]^2\\
-1& \ \ {\rm with \ prob. \ } 2(1-p)\frac{w(1)Z_1(j)}{S(j)}\big[1-\frac{w(1)Z_1(j)}{S(j)}\big]\\
&\ \ \ \ \ \ \ \   + (1-p)\frac{[w(1)]^2Z_1(j)}{[S(j)]^2}\\
-2& \ \ {\rm with \ prob. \ } (1-p)\big[\frac{w(1)Z_1(j)}{S(j)}\big]^2\\
&\ \ \ \ \ \ \ \   - (1-p)\frac{[w(1)]^2Z_1(j)}{[S(j)]^2}.
\end{array}\right.
$$
$$d_2(j+1) \ = \ \left\{\begin{array}{rl}
2& \ \ {\rm with \ prob. \ } (1-p)\big[\frac{w(1)Z_{1}(j)}{S(j)}\big]^2\\
&\ \ \ \ \ \ \ \  - (1-p)\frac{[w(1)]^2Z_1(j)}{[S(j)]^2}\\
1& \ \ {\rm with \ prob. \ } p\frac{w(1)Z_{1}(j)}{S(j)}\\
&\ \ \ \ \ \ \ \   + 2(1-p)\frac{w(1)Z_{1}(j)}{S(j)}\big[1-\frac{w(1)Z_{1}(j)}{S(j)} - \frac{w(2)Z_{2}(j)}{S(j)}\big]\\
0& \ \ {\rm with \ prob. \ } p\big[1-\frac{w(1)Z_{1}(j)}{S(j)} - \frac{w(2)Z_2(j)}{S(j)}\big] \\
&\ \ \ \ \ \ \ \  + (1-p)\frac{[w(1)]^2Z_1(j)}{[S(j)]^2}\\
&\ \ \ \ \ \ \ \  + 2(1-p)\frac{w(1)Z_{1}(j)}{S(j)}\frac{w(2)Z_2(j)}{S(j)}\\
&\ \ \ \ \ \ \ \  + (1-p)\big[1-\frac{w(1)Z_{1}(j)}{S(j)} - \frac{w(2)Z_2(j)}{S(j)}\big]^2\\
-1& \ \ {\rm with \ prob. \ } p\frac{w(2)Z_2(j)}{S(j)} \\
&\ \ \ \ \ \ \ \  + (1-p)\frac{[w(2)]^2Z_2(j)}{[S(j)]^2}\\
&\ \ \ \ \ \ \ \  + 2(1-p)\frac{w(2)Z_2(j)}{S(j)}\big[1-\frac{w(1)Z_{1}(j)}{S(j)} - \frac{w(2)Z_2(j)}{S(j)}\big]\\
-2& \ \ {\rm with \ prob. \ } (1-p)\big[\frac{w(2)Z_2(j)}{S(j)}\big]^2\\
&\ \ \ \ \ \ \ \  - (1-p)\frac{[w(2)]^2Z_2(j)}{[S(j)]^2}.
\end{array}\right. $$

\section{Appendix: Non-compactness of semigroups}
\begin{proposition}
\label{appendix_prop}
The semigroups $P_t$ and $P_t^B$ are not compact for any $t\geq 0$.
\end{proposition}

\begin{proof}
For $x\in \Omega$, let $\zeta(t;x) = P_t^Bx$.  From the form of $B$ (cf. after \eqref{psi_ODE}), we observe that
\begin{eqnarray}
\label{app_1}
\sum_{k=1}^L k\zeta_k(t;x)   &= &  \sum_{k=1}^Lk\zeta_k(0;x) +q_0\int_0^t \sum_{k=1}^{L-1} w(k)\zeta_k(s;x)ds \nonumber \\
&&\ \ \ \ \ \ \ \ \ \ \ \ \ \ \ \
  -q_0Lw(L)\int_0^t\zeta_L(s;x)ds
\end{eqnarray}
 when $x\in \ell_c$ (cf. proof of Proposition \ref{T_bounds}).  

Fix now $x\in \ell_c$ positive.  Since $P_t^B$ is positive (Proposition \ref{positive_prop}), we have $\zeta_k(\cdot;x)\geq 0$ for $k\geq 1$ and 
$\sum_{k=1}^L k\zeta_k(t;x)  \leq  \sum_{k=1}^L k\zeta_k(0;x) +\w q_0\int_0^t \sum_{k=1}^L k\zeta_k(s;x)ds$.  Therefore, the upper bound
$\sum_{k\geq 1}k\zeta_k(t;x)  \leq e^{\w q_0 t}\sum_{k\geq 1} k\zeta_k(0;x)$.
 
We now derive a lower bound.   In \eqref{app_1}, by the upper bound, limits of all terms as $L\uparrow\infty$ converge.  In particular, by positivity, $\sum_{k\geq 1} w(k)\int_0^t \zeta_k(s;x)ds = \int_0^t \sum_{k\geq 1} w(k)\zeta_k(s;x)ds<\infty$, and so the limit $\lim_{L} Lw(L)\int_0^t\zeta_L(s;x)ds = 0$.  Therefore, from \eqref{app_1} and positivity, we get $\sum_{k\geq 1} k\zeta_k(t;x) \geq \sum_{k\geq 1} k\zeta_k(0;x)=\|x\|$.

Let now $t\geq 0$ be fixed.  For $n\geq 1$, let $x^n\in \ell_c$ where $x^n_n = n^{-1}$ and $x^n_k = 0$ for $k\neq n$.  This sequence is bounded in $\Omega$: 
$\|x^n\|  =  \sum_{k\geq 1}k|x^n_k|  =  1$.

Then, starting from $n=1$, let $L_1$ be an index so that $\sum_{k>L^1}k\zeta_k(t;x^1)\leq 1/2$.  For $n\geq 1$, define $L^{n+1}>L^n$ as an index where $\sum_{k>L^{n+1}}k\zeta_k(t;x^{n+1})\leq 1/2$. 

We now show $\|P_t^Bx^n - P_t^Bx^m\|\geq 1$ for all $1\leq m<n$.   By the form of $B$, there is no flow `backwards', that is $\zeta_k(t;x^n) \equiv 0$ for $k<n$.  Write 
$\sum_{k\geq 1}k|\zeta_k(t;x^m) - \zeta_k(t;x^n)| \geq \sum_{k\leq L^m}k|\zeta_k(t;x^m) - \zeta_k(t;x^n)|
= \sum_{k\leq L^m}k\zeta_k(t;x^m)  \geq  \|x^m\| - 1/2 = 1/2$.
Hence, $P_t^B$ cannot be a compact operator for any $t\geq 0$.  

Similarly, $P_t$ cannot be compact for any $t\geq 0$:  Suppose $P_{t_0}$ is compact.  Then, as $P_u$ is bounded for each $u\geq 0$, $P_{t_0+u}$ is compact for $u\geq 0$.  Because $K$ is compact and $B=A-K$, by say the perturbation result Theorem III.1.14(i) in \cite{EN}, $P_{t_1}^B$ for some $t_1\geq 0$ would also be compact, a contradiction.   
\end{proof}


\begin{thebibliography}{99}


\bibitem{Albert-Barabasi-02}
{\sc Albert, R. and Barab\'{a}si, A.-L.} (2002).
\newblock Statistical mechanics of complex networks.
\newblock {\em Rev. Modern Phys.\/} {\bf 74,} 47--97.


\bibitem{A}
{\sc Athreya, K. B.} (2007).
\newblock Preferential attachment random graphs with general weight function.
\newblock {\em Internet Math.\/} {\bf 4,} 401--418.

\bibitem{AGS}
{\sc Athreya, K. B., Ghosh, A. P. and Sethuraman S.} (2008).
\newblock Growth of preferential attachment random graphs via continuous-time branching processes.
\newblock {\em Proc. Indian Acad. Sci. Math. Sci.\/} {\bf 118,} 473--494.

\bibitem{Baladi}
{\sc Baladi, V.} (2000).
\newblock {\em Positive Transfer Operators and Decay of Correlations}.
\newblock Adv. Ser. in Nonlinear Dynamics {\bf 16} World Scientific, Singapore.

\bibitem{Barabasi}
{\sc Barab{\'a}si, Albert-L{\'a}szl{\'o}} (2009).
\newblock Scale-free networks: a decade and beyond.
\newblock {\em Science\/} {\bf 325,} 412--413.

\bibitem{Albert-Barabasi-99}
{\sc Barab\'{a}si, A.-L. and Albert, R.} (1999).
\newblock Emergence of scaling in random networks.
\newblock {\em Science\/} {\bf 286,} 509--512.

\bibitem{barbour}
{\sc Barbour, A.D., Pugliese, A.} (2005).
\newblock Asymptotic behavior of a metapopulation model.
\newblock {\em Ann. Appl. Probab.} {\bf 15} 1306--1338.






\bibitem{BRST-01}
{\sc Bollob\'{a}s, B., Riordan, O., Spencer, J. and Tusn\'{a}dy, G.} (2001).
\newblock The degree sequence of a scale-free random graph process.
\newblock {\em Random Structures Algorithms\/} {\bf 18,} 279--290.

\bibitem{BS}
{\sc Bornholdt, S. and Schuster, H. G.(eds)} (2003).
\newblock {\em Handbook of Graphs and Networks: From the Genome to the Internet}.
\newblock Wiley-VCH, Weinheim, Germany.

\bibitem{Cald}
{\sc Caldarelli, G.} (2007).
\newblock {\em Scale-Free Networks: Complex Webs in Nature and Technology}.
\newblock Oxford University Press, USA.


\bibitem{Chen_Yao}
{\sc Chen, H., Yao, D.D.} (2001)
\newblock {\em Fundamentals of Queuing Networks: Performance, Asymptotics, and Optimization.}
\newblock Applications of Mathematics; Stochastic Modelling and Applied Probability {\bf 46}.
\newblock Springer-Verlag, New York.

\bibitem{CS}
{\sc Choi, J. and Sethuraman, S.} (2011)
\newblock Large deviations for the degree structure in preferential attachment schemes.
\newblock {\em Ann. Appl. Probab.\/} {\bf 23,} 722--763.

\bibitem{CHJ}
{\sc Chung F., Handjani, S. and Jungreis, D.} (2003).
\newblock Generalizations of P\'olya's urn problem.
\newblock {\em Annals of Combinatorics\/} {\bf 7,} 141--153.




\bibitem{Chung-Lu}
{\sc Chung, F. and Lu, L.} (2006).
\newblock {\em Complex Graphs and Networks} vol.~107 of {\em CBMS Regional
  Conference Series in Mathematics}.
\newblock Published for the Conference Board of the Mathematical Sciences,
  Washington, DC.

\bibitem{CH}
{\sc Cohen R. and Havlin S.} (2010).
\newblock {\em Complex Networks: Structure Robustness and Function}.
\newblock Cambridge University Press.

\bibitem{Collevechio}
{\sc Collevechio, A., Cotar, C. and LiCalzi, M.} (2013).
\newblock {On a preferential attachment and generalized P\'olya's urn model.}
\newblock {\em Ann. Appl. Probab. \/} {\bf 23,} 1219--1253.

\bibitem{DN}
{\sc Darling, R.W.R and Norris, J.R.} (2008).
\newblock Differential equation approximations for
Markov chains.
\newblock{\em Probability Surveys\/}
{\bf 5},  37--79.

\bibitem{dereich-2009}
{\sc Dereich, S. and M\"orters, P.} (2009).
\newblock Random networks with sublinear preferential attachment: Degree evolutions.
\newblock {\em Electron. J. Probab.\/} {\bf 14,} 1222--1267.


\bibitem{DM}
{\sc Dorogovtsev, S. N. and Mendes, J. F. F.} (2003).
\newblock {\em Evolution of Networks: From Biological Nets to the Internet and WWW}.
\newblock Oxford University Press, USA

\bibitem{Drinea-Frieze-Mitzenmacher}
{\sc Drinea, E., Frieze, A. and Mitzenmacher, M.} (2002).
\newblock Balls and Bins models with feedback.
\newblock {\em Proc. of the 11th ACM-SIAM Symposium on Discrete Algorithms (SODA)}, 308--315.




\bibitem{Durrett_book}
{\sc Durrett, R.} (2007).
\newblock {\em Random Graph Dynamics}.
\newblock Cambridge U. Press.

\bibitem{EN}
{\sc Engle, K-J., Nagel, R.} (2000).
\newblock {\em One-Parameter Semigroups for Linear Evolution Equations}
\newblock Springer-Verlag, New York.



\bibitem{Foss}
{\sc Foss, S. Konstantopoulos, T.} (2004).
\newblock An overview of some stochastic stability methods.
\newblock {\em J. Oper. Research\/} {\bf 47} 275--303.



\bibitem{KR}
{\sc Krapivsky, P. and Redner, S.} (2001).
\newblock Organization of growing random networks.
\newblock {\em Phys. Rev. E\/} {\bf 63,} 066123-1 -- 066123-14.

\bibitem{thieme1}
{\sc  Martcheva, M.,  Thieme, H. R.} (2008).
\newblock Infinite ODE systems modeling size-structured metapopulations, macroparasitic diseases, and prion proliferation. 
\newblock {\em Structured Population Models in Biology and Epidemiology}
\newblock Lecture Notes in Math {\bf 1936} 51�113, Springer, Berlin.

\bibitem{thieme2}
{\sc Martcheva, M., Thieme, H. R., Dhirasakdanon, T.} (2006)
\newblock Komogorov's differential equations and positive semigroups on first moment sequence spaces.
\newblock {\em J. Math. Biol.} {\bf 53} 642-671.

\bibitem{Mitzenmacher}
{\sc Mitzenmacher, M.} (2004).
\newblock A brief history of generative models for power law and
lognormal distributions.
\newblock {\em Internet Math.} {\bf 1,} 226--251.


\bibitem{Mori-01}
{\sc M\'{o}ri, T.~F.} (2002).
\newblock On random trees.
\newblock {\em Studia Sci. Math. Hungar.\/} {\bf 39,} 143--155.


\bibitem{Newman}
{\sc Newman, M. E. J.} (2003).
\newblock The Structure and Function of Complex Networks.
\newblock {\em SIAM Review\/} {\bf 45,} 167--256 (electronic).

\bibitem{Newman10}
{\sc Newman, M.~E.~J.} (2010).
\newblock {\em Networks: An Introduction}.
\newblock Oxford University Press, USA.


\bibitem{NW}
{\sc Newman, M.~E.~J. and Watts, D.~J.} (2006).
\newblock {\em The Structure and Dynamics of Networks}.
\newblock Princeton University Press, USA.



\bibitem{Oliveira}
{\sc Oliveira, R. and Spencer, J.} (2005).
\newblock Connectivity transitions in networks with super-linear preferential attachment.
\newblock {\em Internet Math.\/} {\bf 2,} 121--163.

\bibitem{PRR}
{\sc Pek\"oz, E.A., R\"ollin, A., Ross, N.} (2013).
\newblock Degree asymptotics with rates for preferential attachment random graphs.
\newblock {\em Ann. Appl. Probab.\/} {\bf 23,} 1188--1218.

\bibitem{Pemantle}
{\sc Pemantle, R.} (2007).
\newblock A survey of random processes with reinforcement.
\newblock {\em Probab. Surv.\/} {\bf 4,} 1--79.


\bibitem{Robert}
{\sc Robert, P.} (2003).
\newblock {\em Stochastic Networks and Queues.}
\newblock Applications of Mathematics; Stochastic Modelling and Applied Probability {\bf 52}.
\newblock Springer-Verlag, Berlin.

\bibitem{RT}
{\sc Rath, B., Toth, B.} (2009).
\newblock Erdos-Renyi random graphs + forest fires = self-organized criticality.
\newblock {\em Elec. J. Probab.\/} {\bf 14} 1290--1327.

\bibitem{Ross}
{\sc Ross, N.} (2012).
\newblock Power laws in preferential attachment graphs and Stein's method for the negative binomial distribution.
\newblock Arxiv preprint arXiv 1208.1558.



\bibitem{RTV}
{\sc Rudas, A., T\'oth, B. and Valk\'o, B.} (2007).
\newblock {Random trees and general branching processes.}
\newblock {\em Random Struct. Algorithms\/} {\bf 31,} 186--202.

\bibitem{Schmidt}
{\sc Schmidt, W.M.} (1980).
\newblock {\em Diophantine Approximation}
\newblock Lecture Notes in Mathematics {\bf 785} Springer-Verlag, Berlin.

\bibitem{Simkin}
{\sc Simkin, M. V. and Roychowdhury, V. P.} (2011).
\newblock Re-inventing {W}illis.
\newblock {\em Phys. Rep.\/} {\bf 502}, 1--35.




\bibitem{Webb}
{\sc Webb, G.} (1985).
\newblock {\em Theory of Nonlinear Age Dependent Population Dynamics}.
\newblock Marcel Dekker, New York.

\bibitem{Whitt}
{\sc Whitt, W.} (2002).
\newblock {\em Stochastic-Process Limits.}
\newblock Springer Series in Operations Research.
\newblock Springer, New York.

\bibitem{W}
{\sc Wimmer, M.} (2006)
\newblock A law of large numbers and central limit theorem for the leaves in a random graph model.
\newblock MS thesis, Department of Mathematics, Iowa State University.


\bibitem{Wormald}
{\sc Wormald, N. C.} (1995).
\newblock Differential equations for random processes and random graphs.
\newblock {\em Ann. Appl. Probab.\/} {\bf 5,} 1217--1235.


\end{thebibliography}
\end{document}